\documentclass[11pt]{amsart}
\usepackage{amssymb, amscd, amsmath, amsthm, epsf, epsfig, latexsym,color} 
\usepackage{hyperref}
\usepackage{float}

\usepackage[T1]{fontenc}
\usepackage{tgschola}

\usepackage{comment}

\usepackage{cutwin}
\usepackage{graphicx}

\usepackage{tasks}
\usepackage{amssymb}
\definecolor{darkblue}{rgb}{0.0,0.0,0.55}

\usepackage{pinlabel}
 \usepackage{tikz-cd}
 \usepackage{wrapfig}
 \usepackage{caption}
 
 \usepackage[inline]{enumitem}

\newcommand{\CP}{\mathbb{CP}}
\newcommand{\bCP}{\overline{\mathbb{CP}}}

\newcommand{\TryPackage}[3]{\IfFileExists{#1.sty}{\usepackage{#1}#2}{#3}
}
\TryPackage{mathrsfs}{\renewcommand{\mathcal}{\mathscr}}{%
        \TryPackage{eucal}{}{}}

\newcommand{\ep}{\epsilon}

\newcommand{\bbi}{{{\bf i}}}
\newcommand{\bbj}{{{\bf j}}}

\newcommand{\ZZ}{{\mathbb Z}}
\newcommand{\RR}{{\mathbb R}}
\newcommand{\CC}{{\mathbb C}}

\newcommand{\NN}{{\mathbb N}}

\newcommand{\TT}{{\mathbb T}}

\newcommand{\sur}{{\rm Sur}}

\def\Nbd#1{{\rm Nbd}({#1})}

\graphicspath{ {figures/} }

\theoremstyle{definition}

\newtheorem{df}{Definition}[section]

\theoremstyle{plain}

\newtheorem{thm}[df]{Theorem}

\newtheorem{cor}[df]{Corollary}
\newtheorem{lem}[df]{Lemma}
\newtheorem{prop}[df]{Proposition}

\numberwithin{equation}{section}

\date{October 17, 2026}

\author{Paul Kirk}

\address{Department of Mathematics, Indiana University \newline
\hspace*{.375in}  Bloomington, IN 47405} 

\email{\rm{pkirk@iu.edu}}

 \opencutright
\renewcommand{\windowpagestuff}

\begin{document}

\title[A fundamental group calculation] {On the  fundamental group calculation associated to  reverse-engineering an exotic $\CP ^2\#3\overline{\CP }^2$ and $\CP ^2\#2\overline{\CP }^2$}

 \maketitle
\begin{abstract} 
I provide  a detailed   exposition
of the fundamental group calculations in the ``reverse engineering'' construction of an exotic $\CP ^2\#3\overline{\CP }^2$   and of an exotic $\CP ^2\#2\overline{\CP }^2$.  \end{abstract}
\setcounter{tocdepth}{1}

\section{Preliminaries}
\subsection{Introduction}   This article contains a detailed  construction of a simply connected homology $\CP ^2\#3\overline{\CP }^2$ and a simply connected homology $\CP ^2\#2\overline{\CP }^2$.  
 
 \medskip

Explicitly
\begin{enumerate}
\item Following \cite{BK}, I construct a closed 4-manifold $Z$ with Euler characteristic $6$ and signature $-2$  and prove in Theorem \ref{thm5.1} that $Z$ is simply connected.
\item Using the surface in $\TT^4\#\bCP^2$ introduced in  \cite{AP2}, I construct a closed 4-manifold $Z_{AP}$ with Euler characteristic $5$ and signature $-1$. Then I prove that $Z_{AP}$ is simply connected in Theorem \ref{Zap} .
\item I isolate the cancellation ``trick'' which underlies these two examples in Theorem \ref{trick}.
\item I give a short proof of a sharpening  of Theorem 2 of \cite{BK} in Theorem \ref{lemDumb}, computing the fundamental group of the complement of a pair of Lagrangian tori in a product of punctured symplectic tori.
\end{enumerate}
  
In   Appendix \ref{exotic}, I prove that $Z$ and $Z_{AP}$ are  exotic smooth structures on $\CP^2\# 3\bCP^2$ and $\CP^2\# 2\bCP^2$.

\medskip

My  goal was to understand  
\cite{AP2} in the same way as I do \cite{BK}.  I found some simplifications and streamlining of the arguments of \cite{BK} along the way, leading to the first nine sections of this article.
I then applied the same approach to the pair I call $(R,\tilde\Sigma_{AP})$ below.  This pair is the building block  discovered by Akhmedov-Park, the key insight  of \cite{AP2}.
  In sum, this article provides an alternative unified  expository option for those who wish to understand the details of the fundamental group calculations
in this approach  to building small exotica.

   I know of three articles 
by authors other than AP or BK which contain proofs of simple connectivity  of minimal symplectic homology $\CP^2\#k\bCP^2$ for $k=2,3$: Fintushel-Stern \cite{FS3}  and  Akbulut \cite{Ak1,Ak2}.  Akbulut constructs Kirby diagrams and combines geometric and algebraic methods to show the manifolds are simply connected, whereas Fintushel-Stern's proof of simple connectivity relies on Theorem 2 of \cite{BK}.

\medskip

The demands on the reader wishing to follow the calculations presented here are familiarity with elementary algebraic  and differential topology,  and the Seifert-Van Kampen (SVK) theorem \cite{Seifert, Van Kampen} in the context of finite CW complexes.   Most of the work is carried out in the fundamental domain $[-\pi,\pi]^n$ for $\TT^n,~n=2,3,4,$ where geometric assertions may be described by illustrations or verified by formulae.  

\medskip

  I have attempted to be methodical and elementary rather than brief: it takes  me about 13 pages   to construct   the homology $\CP ^2\#3\overline{\CP }^2$ and show it is simply connected. The exposition is organized so that the construction and proof of simple connectivity of a homology $\CP ^2\#2\overline{\CP }^2$ is obtained in about eight additional pages.

\subsection{What's different}  Some novel features of the present construction of a homotopy $\CP ^2\#3\overline{\CP }^2$ are: 
We begin with the well-known calculation of the fundamental group of the Borromean rings complement.  We find   {\em complete} group presentations of the fundamental groups of building blocks at each stage of the construction. 
Our presentations are symmetric in  the first  three of the generators   $x,y,z, t$ of $\pi_1(\TT^4)$.  
All six Luttinger surgeries 
are $-1$ surgeries,  and come in pairs exclusively of the form
$(x=\mu_x, t=\mu_y)$ (see Definition \ref{deflut}). We use the same displacement in each of the three pairs of tori (see Figure \ref{fig26fig}).
Each of these choices simplifies the calculations. 

 We construct a convenient parameterization of the tubular  neighborhood  of the  push-off of a framed surface in a 4-manifold in order to accurately compute the fundamental groups of surface sums.  We isolate the algebraic trick which begins the cancellation process in Theorem \ref{trick}. Finally, we choose a slightly non-obvious decomposition in the final step of the construction  
which makes a full presentation easy to compute and then easy to trivialize.

\medskip

The construction of a homotopy $\CP ^2\#2\overline{\CP }^2$ follows the same strategy,   making use of the  symplectic genus two surface   $ \tilde\Sigma_{AP}$ in $\TT^4\#\bCP^2$ constructed in  \cite{AP2}.   


 \medskip

Terminology from symplectic geometry is occasionally inserted to guide the reader who wishes to get from these calculations to exotica, but symplectic geometry plays no logical role in the proofs of Theorems \ref{thm5.1} and \ref{Zap}. Readers  familiar with   symplectic geometry will understand the non-elementary proof, given in Appendix A, that $Z$ and $Z_{AP}$ are not diffeomorphic to a connected sum of another manifold with $\bCP^2$.

\subsection{Contents}
In Sections 2, 3, and 4, a manifold $\TT^4_\sur$ is constructed by performing   torus surgeries on two tori, $T_x,T_y$  in $\TT^4$.  In Section 6, a   closed 4-manifold $P$ containing  a framed genus two surface $\Sigma_P$ is built as the surface sum of two copies of $\TT^4_\sur$ along a torus $T_z$.   Section 7 constructs a  framed genus two surface $\tilde\Sigma_S$  in the manifold  $S=\TT^4_\sur\#2\overline{\CP }^2$.   

 In Section \ref{sec7} a closed 4-manifold $Z$ with the homology of $\CP ^2\# 3\overline{\CP }^2$ is constructed as the surface sum of $S$ and $P$ along $\tilde\Sigma_S$ and $\Sigma_P$. The fundamental group of $Z$ is computed and found to be trivial.

The trick starting  the cancellation process in a presentation of $\pi_1(Z)$ is formalized in 
Theorem \ref{trick}.  We give a short  proof of a strengthening of Theorem 2 of \cite{BK}  in Theorem \ref{lemDumb}.\footnote{Those looking  for an alternative exposition of Theorem 2 of \cite{BK} 
need only read Sections \ref{sect2}, \ref{sect3}, and Theorem \ref{lemDumb}, a total of about four pages.}

In the last section,  we recall the Akhmedov-Park construction of a  genus two surface $\tilde\Sigma_{AP}$ inside $R=\TT^4_\sur\#\overline{\CP }^2$. We then combine Theorem \ref{trick} with some calculations in $\pi_1(R\setminus \tilde\Sigma_{AP})$ to show that the surface sum of $R$ and $P$ is a closed simply connected  homology $\CP ^2\# 2\overline{\CP }^2$.

Appendix \ref{exotic} briefly explains why $Z$ and $Z_{AP}$ are exotic. Appendix \ref{CP} describes, for the convenience of the reader, the four 4-dimensional cut-and-paste constructions of {\em torus surgery, surface sum, resolving a double point, and blowing up}.

 \subsubsection{Study guide} 

A first pass through this article might consist, in order,  of a glance at Figure \ref{fig26fig},  Definition \ref{deflut}, Section \ref{pushoff2}, the statements of Theorem \ref{thmP},  Theorem \ref{thmS},  Theorem \ref{thm5.1},  Figure \ref{fig44fig}, the statement of  Theorem \ref{R}, and finally the statement of Theorem \ref{Zap}.  A second pass would read straight through, avoiding only the proofs of  Lemmata
  \ref{9.2},  \ref{Ball},  and \ref{dicyAF2}.  Armed with a  pot of good coffee and at least five different colored pencils, the reader should at that point   enjoyably  make it through the remaining  proofs.
   
\subsection{Acknowledgements}  Thanks to  T.~Lidman, C.~Livingston, L.~Piccirillo,     and R. Torres for helpful comments and encouragement, and to Cuadros Quartet for inspiration.   Special thanks to   S. Kopylov and O. Korkmaz, who made many valuable suggestions while reading early drafts of this article. 
  Extra special thanks to S. Baldridge for teaching me about this beautiful topic 20 years ago,  for insightful  recent discussions, and for running   an AI proofreader on a late draft.  \footnote{ No AI/LLMs or computer assistants  other than a LaTeX editor were used in producing the August 10, 2026 draft of this article.   Baldridge prompted   ChatGPT to produce a referee report of this draft,  and I  prompted  ChatGPT  to produce a referee report of the arXiv (v1)  version (posted  August 22, 2026),   resulting in the present arXiv (v2)  version. The statements of theorems  and  most details of the arguments remained unchanged throughout this revision process. }     

 \medskip
 
  I hope   readers find the document accurate, complete, and accessible.

   \subsection{Notation and terminology}
In what follows we write $$S^1=\RR/2\pi\ZZ\text{ and }\TT^n=(\RR/2\pi\ZZ)^n.$$
The notation $[a,b]$
means $\{e^{\theta\bbi}\mid a\leq \theta\leq b\}\subset S^1$ and similarly for $(a,b)$. For example, $S^1=[-\tfrac\pi 5,\tfrac\pi 5]\cup[\tfrac\pi 5,2\pi-\tfrac\pi 5]$.    

 \medskip
 
 All surfaces in this article are compact and orientable. The term 
{\em $n$-punctured surface} means the result of removing
$n$ disjoint open disks from a closed surface.

\medskip

The notation $\Nbd{S}$ denotes a tubular neighborhood of a  submanifold $S\subset X$, and $\nu(S,X)$ denotes the normal bundle of $S$, by definition the orthogonal complement to $T_*S$ in $T_*X|_S$ with respect to some Riemannian metric. 

Given a codimension two oriented pair $(X ,S)$, a {\em meridian of $S$} is any oriented (continuous) loop $\mu:S^1\to X \setminus S$ which admits an extension   
$\hat\mu:D^2\to X $   intersecting $S$ in a single point,  transversely and positively.  If $S$ is path connected, every meridian is freely homotopic to the boundary of the normal disk to a point $s\in S$. 
A {\em framing} of $S$ is, by definition, a trivialization of the normal bundle, ${\rm fr}:\nu(S,X )\cong S\times \RR^2$, up to isotopy.  The isotopy class of the normal vector field $\nu:={\rm fr}^{-1}(1,0)$ is determined by and determines ${\rm fr}$.
\medskip

 Henceforth identify $\TT^3=S^1\times S^1\times S^1$ with $ \TT^3\times\{0\}\subset \TT^4$.  The term {\em line segment} will be reserved to mean  a straight line segment in the convex open set $(-\pi,\pi)^4$ in the fundamental domain $[-\pi,\pi]^4$ for $\TT^4.$

 An embedded $k$-torus   in $\TT^n$ is called a {\em coordinate}  torus    provided it is the image of an affine $k$-dimensional subspace of $\RR^n$  which is parallel to one of the coordinate subspaces.

\medskip

 We use two different base points in the 4-torus in this exposition: for the most part we take the base point to be $e=(0,0,0,0)$, but occasionally we use $e_z=(-\tfrac\pi 2 +\tfrac{\ep}{\sqrt{2}},\tfrac\pi 2 - \tfrac{\ep}{\sqrt{2}},0,0)$ (mostly in Section \ref{sec5}). The straight line segment $\rho_1$ joining $e$ to $e_z$ is always used to identify fundamental groups (and their presentations)  based at different points.  

\medskip

We use $\bar g$ for the inverse $g^{-1}$ of a group element $g$. A {\em relation on $g_1,\cdots,g_n$} is a word $w(g_1,\cdots , g_n)$ in the free group on the set $\{ g_1,\cdots,g_n\}$. We often abuse notation and write ``$w_1=w_2$'' for the relation $w_1 \bar w_2$.

\subsubsection{$\epsilon$-$\delta$s} We use two small positive numbers in the following.   First, fix, for the rest of this article, $0<\ep\leq\tfrac{\pi}{30}$.  (Our pictures are somewhat misleading in that to include any detail, $\ep$ appears larger.)  We use  $\ep$  since it appears frequently and takes less typing than $\tfrac{\pi}{30}$.
 
  A  second positive number $\delta<\ep$ will appear below, once in the construction of the second building block, and once (independently) in the construction of the third building block.  
  It will help to assume, in addition to the requirements of each construction on the size of $\delta$, that
  $0<\delta\leq \tfrac{\ep^2}{8}.$   
    
  In any case, the reader should think of $\ep$ geometrically as just small enough so that the $\ep$-neighborhoods in $[-\pi,\pi]^4$ of various disjoint point-sets  which appear in the course of the constructions do not overlap. By contrast, the ratio  $\delta/\ep$   should be thought of as small.

\subsection{The museum genus two surface} Fix henceforth the smooth, closed,  oriented, based, genus two surface $(\Sigma,\sigma)$ obtained by identifying sides labeled by $a_1, b_1, a_2, b_2$ of an  octagon in the usual way, so that
$$\pi_1(\Sigma,\sigma)=\langle a_1, b_1, a_2, b_2\mid [a_1,b_1][a_2,b_2]\rangle\text{ and } a_1\cdot b_1=1.$$

We leave as an exercise to the reader the following elementary
2-manifold lemma, which is used below to parameterize  genus two subsurfaces of 4-manifolds.

\begin{lem} \label{lem3.1} Let $c_1,d_1, c_2,d_2$ be oriented embedded circles in $\Sigma$ so that $c_1$ intersects $d_1$ transversely in one positive point $p_1$, so that $c_2$ intersects $d_2$ transversely in one positive point $p_2$, and all other intersections of $c_i$s  and $d_i$s are empty. Suppose $\rho$ is an arc in $\Sigma$ from $p_1$ to $p_2$ whose interior misses $c_1\cup d_1\cup c_2\cup d_2$ and so that 
$[c_1,d_1][\rho c_2 \bar \rho, \rho d_2 \bar \rho]=1$ in $\pi_1(\Sigma, p_1)$.

Then
 there exists an  orientation preserving diffeomorphism $h: \Sigma \cong \Sigma $ satisfying $h(\sigma)=p_1$, $h(a_1)=c_1$, $h(b_1)=d_1 $, $h(a_2)= \rho c_2\bar\rho$, and $h(b_2)= \rho d_2 \bar \rho$
 in $\pi_1(\Sigma,p_1)$.\qed
\end{lem}

 \section{The fundamental group of the complement of the 3-component link $C$ in $\TT^3$}\label{sect2}

\subsection{Some submanifolds of and loops in $\TT^3$}

Consider the link $C$ of three coordinate circles 
\begin{equation}\label{Cz}C_x=S^1\times\{-\tfrac \pi 2\}\times\{\tfrac \pi 2\},C_y=\{\tfrac \pi 2\}\times S^1\times \{-\tfrac \pi 2\},C_z=\{-\tfrac\pi 2\}\times \{\tfrac \pi 2\}\times S^1
\end{equation}  and the coordinate torus 
\begin{equation}\label{EE}E=\TT^2\times\{0\}
\end{equation} in $\TT^3$.
Set  
\begin{equation}\label{Y}
 Y:=\TT^3\setminus \Nbd{C_z} \text{ and }
 Y^c:=\TT^3\setminus \Nbd{ C_x\sqcup C_y\sqcup C_z},\end{equation} 
 where, to be definite, $\Nbd{C_z}$ and $\Nbd{ C_x\sqcup C_y\sqcup C_z}$ mean  tubular neighborhoods of radius $\ep$.  So $Y$ has one 2-torus boundary component  and $Y^c$ has three. Define $E_0=E\setminus \Nbd{C_z}$.

\bigskip

 \begin{wrapfigure}[8]{r}{0.3\textwidth}
\begin{center}
\vskip-.4in
\includegraphics[width=2.5in]{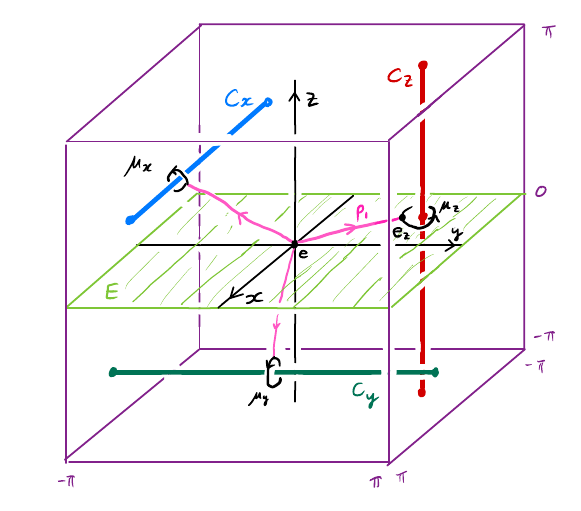}
 \vskip-.2in
\caption{\label{fig26fig}}
\end{center}
\end{wrapfigure}
Figure \ref{fig26fig} shows the link $C$ and the torus $E$ in the fundamental domain $[-\pi,\pi]^3$ for $\TT^3$.  Also depicted are the base point $e=(0,0,0)$,  and based loops
 $x(\theta)=(\theta,0,0),~ y(\theta)=(0,\theta,0),~z(\theta)=(0,0,\theta),~ 0\leq \theta\leq 2\pi.$ 
In addition, based meridians  $\mu_x,  \mu_y,$ and $\mu_z$  for $C$ are shown.  These are based using the indicated line segments,  which lie in the $yz$, $xz$, and $xy$ planes. The line segment $\rho_1(u)=(-u,u,0,0),~0\leq u\leq\tfrac\pi 2 - \tfrac{\ep}{\sqrt{2}}$ is shown.

\bigskip

 \subsection{Fundamental group of $Y^c$}
\begin{prop} \label{thm1}The fundamental group $\pi_1(Y^c,e)$ has the presentation  
$$
\pi_1(Y^c,e) =\langle x,~y,~z
\mid [x,[z, \bar y]], [y, [x,\bar z ]], [z,[y,\bar x]]
\rangle.
$$
In $\pi_1(Y^c,e)$,  the loops $\mu_x,\mu_y,\mu_z$ satisfy
$$ \mu_x=[z, \bar y],~ \mu_y=[x,\bar z ],~\mu_z=[y,\bar x].
$$
\end{prop}
\noindent{\bf Remark}/{\em Proof.} The reader may recognize that the given presentation of $\pi_1(Y^c)$ is the same as the Wirtinger presentation of the fundamental group of the Borromean rings complement. Indeed, 
Proposition \ref{thm1} follows from the well-known  fact that $0$-surgery on the Borromean rings $(S^3,B)$ produces $(\TT^3,C)$.    

Explicitly,  Hui-Purcell    \cite[Section 2]{Hui-Purcell} produce a diffeomorphism of $Y^c$ with the Borromean rings complement $S^3\setminus B$ (which appropriately interchanges meridians and longitudes), from which the proof  of  Proposition \ref{thm1} reduces to the  Wirtinger calculation of   $\pi_1(S^3\setminus \Nbd B)$.  For the reader's convenience, we provide,  in Appendix \ref{prop21}, an  elementary proof  which bypasses the identification with the Borromean rings. \qed

\section{The tori $T_x,T_y,T_z$ in the 4-manifold $Y^c\times S^1$}\label{sect3}

The 4-torus $\TT^4$ contains the two (Lagrangian with respect to the standard symplectic form $dxdy + dzdt$ on $\TT^4$) coordinate tori
\begin{equation}\label{CxCy}T_x=C_x\times S^1=\{(x,-\tfrac\pi 2, \tfrac\pi 2,t)\mid x,t\in S^1\} \text{ and } T_y=C_y\times S^1=\{(\tfrac\pi 2, y, -\tfrac\pi 2,t)\mid y,t\in S^1\}\end{equation}
and the (symplectic) coordinate torus 
$$T_z=C_z\times S^1=\{(-\tfrac\pi 2, \tfrac\pi 2,z, t)\mid z,t\in S^1\}.$$
Notice that
$$Y^c\times S^1=\TT^4\setminus \Nbd{T_x\sqcup T_y\sqcup T_z}.$$
 This compact  4-manifold has boundary a union of three 3-tori, namely the boundaries of the tubular neighborhoods of the three 2-tori $T_x, T_y$ and $T_z$.  Thus torus surgery on $T_x,T_y$ in $Y\times S^1$ corresponds to attaching two copies of $\TT^2\times D^2$ to two of the three boundary components of $Y^c\times S^1$.

\medskip

Denote by $t$ the loop in $Y^c\times S^1\subset \TT^4$ given by $$t(\theta)=(e,\theta),~ 0\leq \theta\leq 2\pi.$$   The based loops $x,y,z,\mu_x,\mu_y,\mu_z$ in $Y$ defined above are considered as based loops in $Y^c\times S^1$ via the inclusion $Y^c=Y^c\times\{0\}\subset Y^c\times S^1$.

\medskip

Set
 \begin{equation}\label{Bor}
 {\rm Bor}=\{[x,[z, \bar y]], [y, [x,\bar z ]], [z,[y,\bar x]] \}\end{equation}
 and  
  \begin{equation}\label{tCen}{\rm tCen}=\{[t,x],[t,y],[t,z]\} .\end{equation}
  These are three-element subsets of the free group on $x,y,z,t$. Proposition \ref{thm1}  implies the following.
\begin{cor}\label{cor2}
$$\pi_1(Y^c\times S^1,e)=
\langle
x,y,z, t
\mid
{\rm Bor, tCen} \rangle,
$$
and in $\pi_1(Y^c\times S^1)$,
$$
\mu_x=[z, \bar y],~ \mu_y=[x,\bar z ],~\mu_z=[y,\bar x].
$$
The three boundary components satisfy
\[\begin{split}
 i_x\big(\pi_1\big(\partial \Nbd{T_x}\big)\big)=\langle x, \mu_x , t \mid  [t,x],~[t,\mu_x],~[x, \mu_x]\rangle,\\
  i_y\big( \pi_1\big(\partial \Nbd{T_y}\big)\big)=\langle y, \mu_y , t \mid  [t,y],~[t,\mu_y],~[y, \mu_y]\rangle,
\\   i_z\big(\pi_1\big(\partial \Nbd{T_z}\big)\big)=\langle z, \mu_z , t \mid  [t,z],~[t,\mu_z],~[z, \mu_z]\rangle.
 \end{split}
 \]
 where $i_x:\partial\Nbd{T_x}\subset Y^c\times S^1$, and similarly for $i_y, i_z$.  These boundary components are based using the shortest line segment to $e$. 
\qed
 \end{cor}

\section{Luttinger surgery on $T_x$ and $T_y$ inside $Y\times S^1$ and $\TT^4$} \label{luts}

We consider only the following  torus surgeries  on $T_x$ and $T_y$.  See Appendix \ref{CP} for the definition of torus surgery in a 4-manifold. 

\begin{df}\label{deflut}Let $(Y\times S^1)_\sur$ denote the smooth compact oriented 4-manifold  with one boundary component obtained from  $Y^c\times S^1$ by 
\begin{enumerate}
\item attaching a copy of $D^2\times \TT^2$ to $Y^c\times S^1$ along the boundary component $\partial \Nbd{T_x}$ in such a way that 
$$\partial (D^2\times\{p\})=x \bar \mu_x$$  in $\pi_1(\partial \Nbd{T_x})\cong H_1(\partial \Nbd{T_x})=\ZZ\langle x, t, \mu_x\rangle$, and
\item attaching a copy of $D^2\times \TT^2$ to $Y^c\times S^1$ along the boundary component  $\partial \Nbd{T_y}$ in such a way that 
$$\partial (D^2\times\{p\})=t \bar\mu_y$$  in $\pi_1(\partial \Nbd{T_y})\cong H_1(\partial \Nbd{T_y})=\ZZ\langle y, t, \mu_y\rangle$.
\end{enumerate}

Denote by $\TT^4_\sur$ the closed 4-manifold obtained from $(Y\times S^1)_\sur$ by replacing $\Nbd{T_z}$ (using the identity gluing map).  
\end{df}

We say $(Y\times S^1)_\sur$  {\em is obtained by Luttinger surgery  on $Y\times S^1$ along $ T_x\sqcup T_y $},   and  $\TT^4_\sur$ {\em is obtained by Luttinger surgery on $\TT^4$ along $ T_x\sqcup T_y $.}  (These two manifolds admit a symplectic structure which agrees with the standard one on $\TT^4$ outside $\Nbd{T_x\sqcup T_y}$  \cite{luttinger}.\footnote{Other torus surgeries produce symplectic manifolds, but this pair,  referred to in the literature as {\em  $-1$ Luttinger surgery along $x$ in $T_x $ and $-1$ Luttinger surgery along $t$ in $T_y $},  suffices for the constructions in this article.})

\medskip

By the SVK theorem, a presentation of the fundamental group of $(Y\times S^1)_\sur$ is obtained from the presentation of $\pi_1(Y^c\times S^1)$ given in Corollary \ref{cor2}
by adding   
\begin{equation}\label{Sur}{\rm Sur}:=\{  x\bar\mu_x ,t\bar\mu_y\}=\{  x[\bar y, z] ,t[\bar z ,x] \}\end{equation}
to the list of relations (see Appendix \ref{CP}). 
 \medskip
 
\noindent{\em Notation.} In order to further consolidate notation, combine (\ref{Bor}), (\ref{tCen}), and (\ref{Sur}) and set 
\begin{equation}\label{Rel}{\rm Rel} :={\rm Bor} \cup {\rm tCen}\cup{\rm Sur},\end{equation}
 an eight-element subset of the free group on $x,y,z,t$.\footnote{ S. Kopylov has observed that Bor follows from  tCen$\cup$Sur.}

\begin{prop} \label{prop2.1}
$$\pi_1((Y\times S^1)_\sur)=\langle x,y,z,t\mid {\rm Rel}\rangle\text{ and }\pi_1(\TT^4_\sur)=\langle x,y,z,t\mid [y,\bar x],~{\rm Rel}\rangle.$$ 
\qed
\end{prop}

\noindent{\bf Remark.} The reader can check that the 
assignment $x\mapsto-\bbj, y\mapsto e^{-\tfrac\pi 4 \bbj},~
z\mapsto -\bbi, ~t\mapsto -1$ defines a non-abelian $SU(2)$ 
representation of $\pi_1(\TT^4_\sur)$.  Hence $\pi_1((Y\times S^1)_\sur)$ and $\pi_1(\TT^4_\sur)$ are non-abelian.
Taking the quotient by the commutator subgroup reveals that
$H_1((Y\times S^1)_\sur)=H_1(\TT^4_\sur)\cong \ZZ^2$, generated by the loops $y$ and $z$, with  $x=0$ and $t=0$. 
 
 \section{Notation, base points, and framing a perturbed surface}\label{sec5}

\subsection{Three copies of $\TT^4_\sur$} A manifold $P$ is constructed below  from two copies  
of $\TT^4_\sur\setminus\Nbd{T_z}$. In a later section, $P$ is joined to $S=\TT^4_\sur\#2\overline{\CP }^2$. To distinguish the required three copies of   $\TT^4_\sur$ we write $$\TT^4_\sur, \TT^{4 \prime}_\sur, \text{ and }\TT^{4 \prime\prime}_\sur,$$  and write $x',y',z',t',\mu_x',\mu_y',\mu_z'$ (resp. $x'',y'',z'',t'',\mu_x'',\mu_y'',\mu_z''$) for the loops $x,y,z,t,\mu_x,\mu_y ,\mu_z $ in $\TT^{4\prime}_\sur$  (resp. $\TT^{4\prime \prime}_\sur $).  Similarly we let $T_x', T_y', T_z' , E'$ (resp.
$T_x'', T_y'', T_z'' , E''$) denote the corresponding tori in  
$\TT^{4 \prime}_\sur$ (resp. $  \TT^{4 \prime\prime}_\sur$), and $\rho_1',\rho_1''$ the corresponding line segments from $e'$ to $e_z'$ and $e''$ to $e_z''$.  

\medskip

The manifold $P$ is constructed using the copies $\TT^{4 \prime}_\sur$ and $\TT^{4 \prime\prime}_\sur$, and the manifold $S$ using $\TT^{4}_\sur$. In Section \ref{DA} below, the manifold $S$ is not used, but a different manifold $R=\TT^4_\sur\#\bCP^2$ is used instead.

\subsection{Moving base points}

The base point $e'$ does not lie on the boundary of $\Nbd{T_z'}$, but $e_z'=(-\tfrac\pi 2 +\tfrac{\ep}{\sqrt{2}},\tfrac\pi 2 - \tfrac{\ep}{\sqrt{2}},0,0) $ does.   The line segment  $\rho_1'$   joins  $e' $ to $e_z'$ in $\TT^{4\prime}_\sur$, and defines the base point moving isomorphism 
$$
\pi_1(\TT^{4\prime}_\sur\setminus \Nbd{T_z'},e'_z)\to 
\pi_1(\TT^{4\prime}_\sur\setminus \Nbd{T_z'},e'),
\gamma\mapsto \rho_1'\gamma\bar\rho_1'.$$
 In this section, we apply this isomorphism implicitly, that is,  we use the notation 
$x',y',z',t', \mu_z' $ instead of  
$\rho_1'x'\bar\rho_1',\rho_1'y'\bar\rho_1',\rho_1'z'\bar\rho_1',t', \rho_1'\mu_z' \bar\rho_1'$. 
The same comment applies to  $e'', e_z''$ and the loops $ x'', y'',\mu_z''$ in $ \TT^{4\prime\prime}_\sur.$

\subsection{Perturbing a framed surface in a 4-manifold}\label{pushoff2}  Suppose that 
$F\subset X$ is a closed oriented smooth surface
in an oriented riemannian 4-manifold $X$ satisfying $F\cdot F=0$.

A  vector field $\nu:F\to T_*X$ along $F$ so that its orthogonal projection to $\nu(F,X)=T_*F^\perp$ is nowhere zero is called {\em a transverse vector field to $F$} below.   For any $\delta>0$,  set 
$$F_{\nu,\delta}=\exp(\delta\tfrac{\nu}{|\nu|}(F)),$$ the  {\em $\delta$-push-off of  $F$ in $X$  using $\nu$}. When $\nu$ is clear from context we write $F_\delta$.  The tubular neighborhood theorem implies that there exists a $\delta_0$ so that $F_\delta$ is embedded for all $0<\delta\leq \delta_0$.

In the following lemma,  $D^2$ means the unit disk in $\CC$.

\begin{lem}\label{pushoff} Given $F\subset X $  a closed oriented surface in an oriented  riemannian 4-manifold and a transverse vector field $\nu$ to $F$, then for all  $\delta$ small enough, there exists a  smooth embedding $\tilde h:F\times D^2\hookrightarrow X$ so that 
\begin{enumerate}
\item $\tilde h(F,0)=F_{\nu,\delta}$,  and hence $\tilde h(F\times D^2)$ is a tubular neighborhood of $F_{\nu,\delta}$.
\item  $\tilde h(f,1)=f $ for $f\in F$, and hence $F$ lies on the boundary of the tubular  neighborhood of  $F_{\nu,\delta}$. \end{enumerate}
 \end{lem}
 \begin{wrapfigure}[10]{r}{0.2\textwidth}
\begin{center}
\vskip-.3in
\includegraphics[width=2.2in]{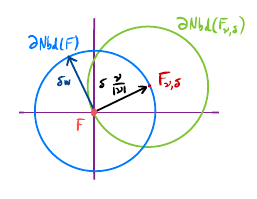}
 \vskip-.1in
\caption{\label{fig54fig}}
\end{center}
\end{wrapfigure}
\noindent{\em Proof.} 
Choose a unit length vector field $w:F\to T_*X$ along $F$ so 
that $(\tfrac{\nu}{|\nu|},w)$ is pointwise orthonormal and $(T_fF,  \nu(f),w(f))$ properly orients $T_fX$.

 Set $W\to F$ to be the trivial 2-dimensional subbundle of $ T_*X|_F$ spanned by $\nu $ and $w$. Hence $T_*X|_F=T_*F\oplus W$, and W is isomorphic to the normal bundle of $F$ in $X$.
Choose $\delta$ small enough so that $\exp$ embeds the $3\delta$ neighborhood  of the zero section of $W$. 
Then set
$$\tilde h(f, re^{\theta\bbi})=\exp\big(
\delta(1-r\cos\theta) \tfrac{\nu(f)}{|\nu(f)|} - \delta r\sin\theta w(f)
\big)
$$
 (see Figure \ref{fig54fig}). 
Since $\tilde h(f,1)=f$ and $\tilde h (f,0)=\exp(\delta\tfrac{\nu(f)}{|\nu(f)|})$, it follows that 
$\tilde h (F\times D^2)$ is a tubular neighborhood of $F_{\nu,\delta}$ satisfying the conclusions of the lemma.
\qed
\medskip 
 
 The exponential map for $\TT^4=(\RR/2\pi\ZZ)^4$ is given simply by $\exp(v,w)=v+w$.
In what follows, the riemannian 4-manifolds $X$ we consider contain  an open set $U$ which is also an open subset of $\TT^4$. This implies that  if $F$ is a surface in $X$, then, away from the boundary of $U$, 
$$(F\cap U)_{\nu,\delta}= \{ f + \delta\tfrac{\nu(f)}{|\nu(f)|}\mid f\in F\cap U\}.$$

\section{The surface sum $P$ of $\TT^{4 \prime}_\sur$  and $\TT^{4\prime\prime}_\sur$ along the tori $T_z'$ and $T_z''$} 

We construct, in Section \ref{bbP} below, a {\em  building block} consisting of a pair $(P,\tilde h_P:\Sigma\times D^2\to P)$, where 
$P$ is the surface sum (see Appendix \ref{CP}) of two copies of $\TT^4_\sur$ along their tori $T_z$, and $\tilde h_P$ is a framed embedding of $\Sigma$ into $P$.

\medskip

 In what follows,  a subset of $\TT^4$ is considered as a subset of  $\TT^4_\sur$ if it  misses $\Nbd{T_x\cup T_y}$.

\subsection{The sum of two copies of $\TT^4_\sur$ along $T_z$} The boundary of $\Nbd{T'_z}$ can be explicitly parameterized as
$$\partial  \Nbd{T_z'}=\{ (-\tfrac\pi 2, \tfrac\pi 2, z, t) +\ep(\cos(\theta-\tfrac{\pi}{4}),\sin(\theta-\tfrac{\pi}{4}),0,0)\mid  z, t,\theta\in S^1\},$$
so that $e'_z$ corresponds to $(z,t,\theta)=(0,0,0)$ and $\partial E_0'$ corresponds to $(z,t,\theta)=(0,0,\theta)$. Similar comments apply to $T_z''$.

The linear isomorphism 
\begin{equation}\label{phi2}\phi(z,t,\theta)=(t, -z, -\theta), \end{equation}   and these parameterizations induce  an orientation reversing diffeomorphism
 $$\phi:\partial\Nbd{T_z'}\to \partial\Nbd{T_z''}$$
 which takes $e_z'$ to $e_z''$ and induces
\begin{equation}\label{phi} \phi(\mu_z')=\bar\mu_z'',  ~\phi(z')=\bar t'', ~ \phi(t')= z'' \end{equation}
on (the abelian) fundamental groups.\footnote{One can take $\phi$ to be the block sum of $(-1)$ and any matrix in $SL_2(\ZZ)$ and obtain a symplectic manifold. The choice (\ref{phi}) suffices for the constructions of this article.}
By construction,  $\phi(\partial E_0')=\overline{\partial E_0}''$.

\begin{df}\label{P}Define the {\em  surface sum  of $\TT^{4 \prime}_\sur$  and $ \TT^{4 \prime\prime}_\sur$ along $T_z'$ and $T_z''$} to be the closed (symplectic) manifold
 \begin{align*}
 P &:=\TT^{4 \prime}_\sur\#_{\phi} \TT^{4 \prime\prime}_\sur  =\big(\TT^{4 \prime}_\sur \setminus\Nbd{T_z'}\big)\cup_{\phi}\big(\TT^{4 \prime\prime}_\sur \setminus\Nbd{T_z''}\big).\end{align*}
\end{df}

\begin{lem}
 \label{lem4A} $\pi_1(P, e_z')$ has the presentation 
 $$
 \langle z', x' ,y',t' ,
 z'', x'', y'',t''
 \mid  {\rm Rel}',~   {\rm Rel}'',~{\rm Gl}  \rangle,
$$
 where \begin{equation}\label{Gl}
{\rm Gl}=\{ [y',\bar x'][y'',\bar x''], ~  z' t'', ~ t'\bar z''  \}.\end{equation}
\end{lem}
\begin{proof}
 SVK theorem.
\end{proof}

It follows that $H_1(P)=\ZZ\langle y', y''\rangle$, since Sur$'$, Sur$''$ imply that, modulo the commutator subgroup,
$x'=1, t'=1, x''=1, t''=1$
and using Gl, $ z'=\bar t''=1, z''= t'=1$.

\subsection{The surface $\Sigma_P$ and its framing}
Since the gluing diffeomorphism $\phi$ of 3-tori  satisfies $\phi(\partial E_0')=\overline{\partial E_0}''$,    $P$ contains the smooth closed oriented genus two surface $$\Sigma_P:=E_0^{'}\cup_{\phi|_{ \partial E_0' }}E_0^{''}.$$ 
This surface is illustrated in Figure \ref{fig62fig}.
\begin{figure}[H]
 \labellist
 \small\hair 2pt
  \endlabellist
\centering
\includegraphics[width=3.8in]{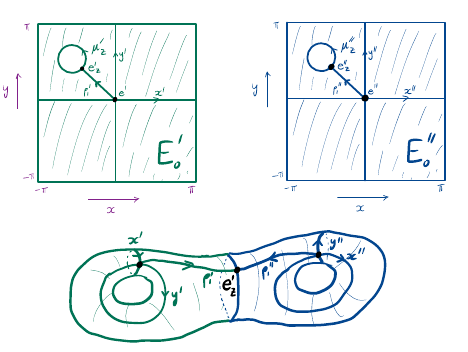}
 \vskip-2ex
 \caption{\label{fig62fig} The surface $\Sigma_P$ }\end{figure}

The coordinate tangent vectors $\tfrac{\partial}{\partial x},~
\tfrac{\partial}{\partial y},~\tfrac{\partial}{\partial z},$ and $\tfrac{\partial}{\partial t}$ form an orthonormal basis
at the tangent space of any point in $\TT^4$.

\begin{lem}The surface $\Sigma_P$ in $P$ admits a transverse  vector field
 $ \nu_P\colon \Sigma_P\to T_*P$ given by 
 \begin{equation}
 \nu_P(s)=\begin{cases}  \tfrac{\partial}{\partial z}+\tfrac{\partial}{\partial t} & \text{ if } s\in E_0'\\
   \tfrac{\partial}{\partial z}-\tfrac{\partial}{\partial t} & \text{ if } s\in E_0''.
\end{cases} 
 \end{equation}  \end{lem}
 \begin{proof} The gluing diffeomorphism
 $\phi:\TT^3\to \TT^3$ (\ref{phi2}) satisfies $d\phi(\tfrac{\partial}{\partial z}+\tfrac{\partial}{\partial t})=\tfrac{\partial}{\partial z}-\tfrac{\partial}{\partial t}$, and so the two definitions agree on the overlap.
 \end{proof}

 Choose $0<\delta<\tfrac{\ep^2} 8$ small enough so that the $\delta$-push-off,  $\Sigma_{P,\delta}$, of $\Sigma_P$ using $\nu_P$ is embedded. 
Then $$\Sigma_{P,\delta}=E_{0,\delta}'\cup E_{0,\delta}'',$$ where, since $Y^c\times S^1\subset \TT^4,$ $$
E_{0,\delta}'=E_0\times \{(\tfrac{\delta}{\sqrt{2}},\tfrac{\delta}{\sqrt{2}})\}\text{ and }E_{0,\delta}''=E_0\times \{( \tfrac{\delta}{\sqrt{2}},-\tfrac{\delta}{\sqrt{2}})\}$$
(see Section \ref{pushoff2}).

   \begin{cor}
  The loops $x', y', z', t', \mu_x',\mu_y',\mu_z'$ in $\TT^{4\prime}_\sur $ and $x'', y'', z'', t'', \mu_x'',\mu_y'',\mu_z''$ in $\TT^{4\prime\prime}_\sur $
 are disjoint from   $\Sigma_{P,\delta}$.
  \end{cor}
  
\begin{proof} With the exception of $t'$ and $t''$, these loops all have fourth coordinate equal to $0$. On the other hand  each half of 
$\Sigma_{P,\delta}=E'_{0,\delta}\cup E''_{0,\delta}$ has non-zero fourth coordinate. 
Similarly,  $t'$  and $t''$ have zero third coordinate and hence miss $E_{0,\delta}'$ and $E_{0,\delta}''$.
\end{proof}

\begin{lem} \label{lem1.3} There is a smooth orientation preserving embedding $\tilde h_P \colon  \Sigma \times D^2 \hookrightarrow P$  such that 
\begin{enumerate}
\item $\tilde h_P(\Sigma\times 1)=\Sigma_P$.
\item $\tilde h_P(\Sigma\times 0)=\Sigma_{P,\delta}$.
\item $\tilde h_P(\sigma,1)=e_z'$.
\item The composite $$ \pi_1(\Sigma,\sigma)\xrightarrow{\tilde h_P(-,1)}
\pi_1(\Sigma_P, e_z')
\xrightarrow{i_P} 
\pi_1(P, e_z') 
$$ is given by
$$  
a_1\mapsto y', ~ b_1\mapsto  \bar x', ~ a_2\mapsto y'', ~ b_2\mapsto  \bar x''. 
$$
\end{enumerate}

\end{lem}
\begin{proof} The loops $x',y'$ (resp. $x'', y''$)  based at    $e_z'=e_z''$ lie on the closed genus two surface
 $ \Sigma_P,$
and satisfy    $  [y',\bar x'][y'',\bar x'']=1$ and 
$  y'\cdot \bar x'=1=y''\cdot \bar x''$.
 Apply Lemma \ref{lem3.1} to find a diffeomorphism $h_P:(\Sigma,\sigma)\to (\Sigma_P,e_z')$ satisfying $h_P(a_1)=y', ~h_P(b_1)= \bar x', ~h_P(a_2)=y'',h_P(b_2)=\bar x''$.  
  Equip $\Sigma_P$ with the normal vector field $\nu_P$.  Lemma  \ref{pushoff} 
 shows that $h_P$ extends to an embedding $\tilde h_P:\Sigma\times D^2\hookrightarrow P$ satisfying  (1), (2), (3), and (4).
\end{proof}

Define $\mu_{P,\delta}=\tilde h_P(\sigma\times S^1)$.  Hence 
$\mu_{P,\delta}$ is a meridian of $\Sigma_{P,\delta}$.

\subsection{The building block $(P, \tilde h_P)$}\label{bbP}

 We summarize the properties established above about $$\big(P, \tilde h_P\colon \Sigma\times D^2\hookrightarrow P\big).$$
 
  \begin{thm}\label{thmP}
  The pair $(P, \tilde h_P)$ satisfies \begin{enumerate}
\item $\pi_1(P,e_z')=
 \langle z', x' ,y',t' ,
 z'', x'', y'',t''
 \mid  {\rm Rel}',~   {\rm Rel}'',~{\rm Gl}  \rangle
$ and $H_1(P)\cong \ZZ^2$, generated by $y', y''$.
 \item $\tilde h_P(\Sigma\times\{1\})=\Sigma_P$.
 \item $\tilde h_P(\Sigma\times\{0\})=\Sigma_{P,\delta}.$
 \item   $\Sigma_{P,\delta}=E_{0,\delta}'\cup E_{0,\delta}''=
 E_0\times \{(\tfrac{\delta}{\sqrt{2}},\tfrac{\delta}{\sqrt{2}})\}\cup E_0\times \{( \tfrac{\delta}{\sqrt{2}},-\tfrac{\delta}{\sqrt{2}})\}$.

  \item The base points $e_z'=e_z''$,  $e'$ and $e''$, as well as the line segments $\rho_1',\rho_1''$ joining $e'$ to $e_z'$ and $e''$ to $e_z''$,     lie  on $\Sigma_P\subset\partial \Nbd{\Sigma_{P,\delta}}$.
 \item  The loops $x',y',  x'', y''$ and $\mu_z'=\bar \mu_z''$  lie  on $\Sigma_P$ and hence in $P \setminus  \Sigma_{P,\delta} $. Moreover $$[y',\bar x'][y'',\bar x'']=1.$$
  \item The loops $z ', z'',  t',   t'',$  and the meridian $\mu_{P,\delta}$ lie in $P \setminus  \Sigma_{P,\delta}$.
\item The composite $$ \pi_1(\Sigma,\sigma)\xrightarrow{h_P}
\pi_1(\Sigma_P,e_z')\xrightarrow{i_P} 
\pi_1(P, e_z')$$ is given by
$$ 
a_1\mapsto y', ~ b_1\mapsto \bar x', ~ a_2\mapsto y'', ~ b_2\mapsto \bar x ''.
$$
\end{enumerate}\qed
\end{thm}
The following lemma is used in the proof of Theorem \ref{trick}.
\begin{lem} \label{Pcomp} The relations 
 $  z'=\bar t'', ~[t'', y'']=1, ~[t'',x'']=1\text{ and }x'=[z',\bar y']$  hold in $\pi_1(P\setminus \Sigma_{P,\delta})$.
\end{lem}
\begin{proof}  First, the torus $S^1\times\{(0,0)\}\times S^1$ misses $E_{0,\delta}''$ in $\TT^{4\prime\prime}_\sur\setminus \Nbd{T_z''},$  since the third coordinates differ (see Theorem \ref{thmP} (4)). 
This proves that $[t'',x'']=1$ in $\pi_1(\TT^{4\prime\prime}_\sur\setminus \Nbd{E_{0,\delta}'' \cup T_z''}),$ and hence also in $\pi_1(P\setminus \Sigma_{P,\delta})$.  Similarly, the torus $\{0\}\times S^1\times\{0\}\times S^1$ misses $E_{0,\delta}''$ and hence $[t'',y'']=1$ in $\pi_1(P\setminus \Sigma_{P,\delta})$.

The gluing relation   $z'=\bar t''$    holds in $\pi_1(P\setminus \Sigma_{P,\delta})$ because the loop $z'$ lies in   the boundary of  $\TT^{4\prime}_\sur\setminus \Nbd{E_{0,\delta}'\cup T_z'}$   and the loop $t''$ lies in   the boundary of  $\TT^{4\prime\prime}_\sur\setminus \Nbd{E_{0,\delta}''\cup T_z''}$ by Theorem \ref{thmP} (4), and these loops are identified by the gluing map $\phi$.

The annulus 
 $$A_1:=\{(x, -u, u,0)\mid x\in S^1, 0\leq u\leq \tfrac\pi 2-\tfrac{\ep}{\sqrt{2}}\}\subset \TT^{3'}\times\{0\}\hskip2.5in$$  
 has distance at least $\ep$ from $T_x'=\{(x,-\tfrac\pi 2,\tfrac\pi 2,t)\},$ distance at least $ \tfrac{\pi}{2}$ from $  T_y'=\{(\tfrac\pi 2,y, -\tfrac\pi 2,t)\}$, and distance at least    $\tfrac\pi 2$ from $T_z'=\{(-\tfrac\pi 2,\tfrac\pi 2,z,t)\}$.  Also, 
 $A_1$  meets $E_0'$ precisely in the loop $x'$ and hence misses $E_{0,\delta}'$. Hence $A_1$ misses $E_{0,\delta}'$,  from which it follows by the surgery relation that $x'=\mu_x'$ in $\pi_1(\TT^{4'}_\sur \setminus E_{0,\delta}')$. Moreover, the torus $\{0\}\times \TT^2\times \{0\}$ in $\TT^{4\prime}_\sur$
misses $\Nbd{E'_{0,\delta}\cup T_z'}$, from which one deduces that $\mu_x'=[z',\bar y']$ in  $\pi_1(\TT^{4'}_\sur \setminus \Nbd{E_{0,\delta}'\cup T_z'})$ 
 (see Figure \ref{fig59fig}). Therefore
$x'=\mu_x'=[z',\bar y']$ in $\pi_1(\TT^{4'}_\sur \setminus \Nbd{E_{0,\delta}'\cup T_z'})$ and hence also in $\pi_1(P\setminus \Sigma_{P,\delta})$.
\end{proof}
\section{The marked  genus two surface $\Sigma_S\subset \TT^4_\sur $}\label{sect5}

In this section we produce a second building block, which consists of the closed 4-manifold $S=\TT^4_\sur\# 2\overline{\CP }^2$  and an embedding $\tilde h_S\colon \Sigma\times D^2\to S$.

\subsection{The genus two resolution $\Sigma_S$ of $T_z\cup E$}\label{resolution}

The pair of  oriented tori $T_z=C_z\times S^1$ and $E$ intersect at the single point $p=(-\tfrac\pi 2,\tfrac\pi 2,0,0)$, transversely,  and  $T_z\cdot E=1$.  

  \medskip

The {\em oriented resolution} $\Sigma_S$ of $T_z\cup E$   is  defined to be the smooth  oriented closed  genus two subsurface constructed  by removing a pair of    2-discs of radius $\ep$ centered at $p$, one from $T_z$ and  one from $E$,     and replacing them with an embedded annulus $A$ in $B^4_{\ep}(p)\setminus B^4_{\ep/4}(p)$ with the same boundary as the two removed discs.  In fact,  $A$ is obtained by pushing a smooth annulus in  $S^3_{\ep}(p)$ with boundary the Hopf link into $B^4_{\ep}(p)$; see Appendix \ref{CP}.  

Hence 
\begin{equation}\label{eq3.6}\Sigma_S=E_0\cup A\cup T_{z,0},\end{equation}
where $T_{z,0}=T_z\setminus B^2_\ep(p)$
is a once-punctured torus.
The base point $e$ and the loops $x,y, \mu_z$ lie on $\Sigma_S$.   See Figure \ref{fig37fig}.

\subsection{The paths $\rho_i$ in $Y^c\times S^1$} 
 \label{rhoi}
 Recall that $e_z:=(-\tfrac\pi 2 + \tfrac{\ep}{\sqrt{2}},\tfrac\pi 2 - \tfrac{\ep}{\sqrt{2}},0,0)$ is the closest
point on $\partial\Nbd{T_z}$ from 
$e=(0,0,0,0)$.  
 Define the paths:  
\begin{enumerate} 
\item $\rho_1$ is  the line segment from $e=(0,0,0,0)$ to $e_z$. The path $\rho_1$ lies on $E_0\subset \Sigma_S$.   
\item $\rho_2$ is any embedded arc in the annulus $A$ inside $B^4_\ep(p)$ from  $e_z$ to  $(-\tfrac{\pi }{2} ,\tfrac{\pi }{2}, -\tfrac{\ep }{\sqrt{2}},-\tfrac{\ep}{\sqrt{2}})$.  The path $\rho_2$ lies on $A\subset \Sigma_S$.  
\item $\rho_3$ is the line segment  from the endpoint  of $\rho_2$ to $(-\tfrac{\pi }{2} ,\tfrac{\pi }{2}, -3\ep,-3\ep)$.  The path $\rho_3$ lies on $T_{z,0}\subset \Sigma_S$. \end{enumerate}
 The paths $\rho_1,\rho_2,\rho_3$  can be seen in Figures \ref{fig26fig}, \ref{fig37fig}, \ref{fig40fig}, and \ref{fig45fig}.   
  The path $\rho_2$ in the annulus $A$ is not unique, even up to homotopy, and so we fix a choice for $\rho_2$ once and for all.  Note that $\rho_2$ and $\rho_3 $   lie in the $5\ep$ neighborhood  of $\rho_1$, a 4-ball in the fundamental domain $[-\pi,\pi]^4$ which misses the tori $T_x$ and $T_y$.

\medskip

\subsection{Marking $\Sigma_S$}
The base point $e$ and loops $x,y$ lie on $E_0\subset \Sigma_S$. The  loops $x$ and $y$ are embedded circles and intersect transversely and positively at $e$.  Define two oriented, embedded circles $\tilde \zeta$ and $\tilde \tau$ in $T_{z,0}$ by
\begin{equation}\label{tildezetatau}
\tilde\zeta=\{ (-\tfrac\pi 2,\tfrac\pi 2)\} \times S^1\times\{ -3\ep \}\text{ and }
\tilde\tau=\{(-\tfrac\pi 2,\tfrac\pi 2, -3\ep)\}\times S^1.
\end{equation}
Then $\tilde\zeta\cdot \tilde\tau=1$.

Connect $\tilde \zeta$ and $\tilde \tau$ to the base point $e$ using the path   $\rho=\rho_1\rho_2\rho_3$  in $\Sigma_S$ starting at $e$ and ending at $ \tilde\zeta\cap \tilde\tau$. 
 The interior of $\rho$ misses
 $x\cup y\cup   \tilde\zeta\cup \tilde\tau$.   Define the corresponding based   loops in $\Sigma_S$ by 
  $$\zeta:=\rho \tilde\zeta\bar\rho \text{ and }\tau:=\rho\tilde\tau\bar\rho.$$

\begin{prop} \label{lem1.33}  The loops $x,y,\zeta,\tau$ satisfy
  $$[\zeta,\tau][y,\bar x]=1\text{ in } \pi_1(\Sigma_S,e), \text{ and }
 i_S(\tau)=t,  i_S(\zeta)=z\text { in }\pi_1(\TT^4_\sur,e).$$
 There is an orientation preserving diffeomorphism $h_S\colon (\Sigma,\sigma)\cong (\Sigma_S,e)$  such that $h_S \colon \pi_1(\Sigma,\sigma)\to \pi_1(\Sigma_S,e)$ is given by
$$
a_1\mapsto \zeta, ~ b_1\mapsto \tau, ~ a_2\mapsto   y, ~ b_2\mapsto \bar x.
$$
\end{prop}
\noindent{\bf Remark.} Finding a path $\rho$ in $\Sigma_S$ so that
$[\zeta,\tau][y,\bar x]=1$ in $\pi_1(\Sigma_S,e)$ is straightforward, as is finding a path $\rho$ so that
$ i_S(\tau)=t,  i_S(\zeta)=z\text { in }\pi_1(\TT^4_\sur,e).$
The point of this lemma is that the  path $\rho_1\rho_2\rho_3$  satisfies both requirements.
 
\begin{proof}

The choice of $\tilde\zeta$, $\tilde\tau$, and $\rho_3$ guarantees that the   line segment $\rho_3$ from $\partial T_{z,0}$ to $\tilde\zeta\cap\tilde\tau=(-\tfrac\pi 2,\tfrac\pi 2,-3\ep,-3\ep) $ lies on $\Sigma_S$ and bisects the angle in $T_*\Sigma_S$ formed by the (positive) tangent vectors  of $\tilde\zeta$ and $\tilde\tau$ at the intersection point, as indicated on the right and the bottom  in Figure \ref{fig37fig}.

Since $\rho:=\rho_1\rho_2\rho_3$ lies in the convex $5\ep$ neighborhood of $\rho_1,$ $$  i_S(\tau)=t\text{ and } i_S(\zeta)=z\text{ in } \pi_1(\TT^4_\sur,e).$$
Since $\rho_3$  bisects the angle formed by the  tangent vectors  of $\tilde\zeta$ and $\tilde\tau$ at its endpoint, 
$$[y,\bar x]=\rho_1 (\partial E_0)\bar\rho_1=\rho_1 \rho_2(\partial T_{z,0})^{-1}\bar\rho_2\bar\rho_1=\rho[\tilde\tau,\tilde\zeta]\bar\rho=[\tau,\zeta],
$$
so that $[\zeta,\tau][y,\bar x]=1.$ Applying Lemma \ref{lem3.1} completes the proof.
\end{proof}  

\begin{figure}
 \labellist
 \small\hair 2pt
  \endlabellist
\centering
\includegraphics[width=4.8in]{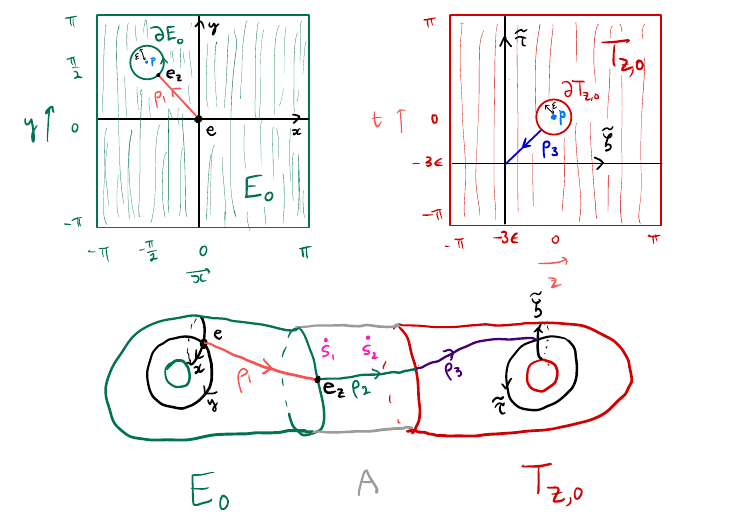}
 \vskip-2ex
 \caption{\label{fig37fig} The surface $\tilde\Sigma_S$.  
\color{red}  }\end{figure}

\subsection{Blowing up $\TT^4_\sur$ twice} 

Note that $\Sigma_S\cdot \Sigma_S=(E+T_z)\cdot(E+T_z)=2$ in $\TT^4_\sur$.
 
\begin{df}\label{S}  Denote 
$$
 S = \TT^4_\sur \# 2\overline{\CP }^2, 
$$
the closed (symplectic) 4-manifold obtained by   Luttinger surgery on $T_x$ and   $T_y$ in $\TT^4$, followed by blowing up $\TT^4_\sur$ twice at points $s_1,s_2\in \Sigma_S$. For convenience, choose $s_1,s_2$ to lie in the annulus $A$ inside $B_\ep(p)$. \end{df}

The proper transform
$\tilde \Sigma_S\subset S$ of $\Sigma_S\subset \TT^4_\sur$ (see Appendix \ref{CP}) equals $ E_0\cup \tilde A\cup T_{z,0},$
where $\tilde {A}$ denotes the proper transform of $A$.
By construction,  $\tilde \Sigma_S\cdot  \tilde \Sigma_S=0$.
Denote the two exceptional 2-spheres by $\mathfrak{E}_1,\mathfrak{E}_2$.

The blow-down map $c:S=\TT^4_\sur\#2\bCP^2\to \TT^4_\sur$ induces an isomorphism on fundamental groups, and restricts to a diffeomorphism $\tilde \Sigma_S\to\Sigma_S$.  Hence Lemma \ref{lem1.33} holds verbatim with $\TT^4_\sur$ replaced by $S$, $\Sigma_S$ replaced by $\tilde\Sigma_S$ and the marking $\Sigma\to \Sigma_S$ replaced by $\Sigma\to \Sigma_S\xrightarrow{c^{-1}}\tilde\Sigma_S$.

\begin{prop}
 The inclusion $\pi_1\big(S \setminus  \tilde\Sigma_S ,e\big)\to \pi_1(S,e)$
 is an isomorphism.
\end{prop}
\begin{proof} $\tilde\Sigma_S$ has a geometrically dual embedded 2-sphere $\mathfrak{E}_1$ in $S.$\end{proof}

\subsection{Framing $\tilde \Sigma_S$} 
The surface $\tilde \Sigma_S=E_0\cup \tilde A\cup T_{z,0}$ admits a transverse vector field $\nu_S$ satisfying
\begin{equation}
 \nu_S(s)= \tfrac{\partial}{\partial z}+\tfrac{\partial}{\partial t}\text{ if } s\in E_0.  
 \end{equation}
This is because every map from the once-punctured torus $E_0$ to $SO(2)$ is nullhomotopic on the boundary circle.   

Choose $0<\delta<\tfrac{\ep^2}{8}$ small enough so that  
$$ \tilde \Sigma_{S,\delta}:= \exp(\delta\cdot \tfrac{\nu_S}{|\nu_S|}(\tilde\Sigma_S)) $$
is an embedded genus two surface.  The subsurface $E_{0,\delta}\subset  \tilde \Sigma_{S,\delta}$ equals $E_0\times \{(\tfrac{\delta}{\sqrt{2}},\tfrac{\delta}{\sqrt{2}})\}$, and in particular misses the loops $x, y, z, t, \mu_x,\mu_y,\mu_z$ in $\TT^{4}_\sur $.  Since the distance from the loops $z$ and $t$ to $T_z$ is greater than $\ep$, the loops $z$ and $t$ lie in $S\setminus    \tilde \Sigma_{S,\delta}$.

Define  
\begin{equation}\label{mudel}\mu_{S,\delta}(\theta)=(0,0, \tfrac{\delta}{\sqrt{2}},\tfrac{\delta}{\sqrt{2}})+\delta\big(0,0,\cos(\theta+\tfrac{5\pi }4), \sin(\theta+\tfrac{5\pi }4)\big),~  \theta\in S^1.\end{equation}
Then  $\mu_{S,\delta}(\theta)$ is the meridian of $\tilde \Sigma_{S,\delta}$ at $e$.

Lemma  \ref{pushoff} implies that surface $\tilde \Sigma_{S,\delta}$  admits a framing
$\tilde h_S:  \Sigma\times D^2\hookrightarrow S$
with image the tubular neighborhood $\Nbd{\tilde\Sigma_{S,\delta}}$ such that 
\begin{enumerate}
\item $ \tilde h_S(\Sigma \times \{0\})=\tilde\Sigma_{S,\delta}$, 
\item $ \tilde h_S(\Sigma \times \{1\})=\tilde\Sigma_{S}$, 
\item $\tilde h_S(\sigma,1)=e$.
 
\item  $\tilde h_S(\sigma,e^{\theta\bbi})=\mu_{{S,\delta}}(\theta).$  \end{enumerate}

For $\delta$ small enough, the exceptional 2-sphere $\mathfrak{E}_1$ intersects  $\tilde\Sigma_{S,\delta}$ transversely in one point, since it intersects $\tilde\Sigma_S$ transversely in one point.     Since $\tilde\Sigma_{S,\delta}$  is path connected,  one can isotope $\mathfrak{E}_1$  if needed so that 
$\mathfrak{E}_1\cap \Nbd{\tilde \Sigma_{S,\delta}} =\tilde h_S(\sigma\times D^2)$. Therefore $\mu_{S,\delta}$  bounds an embedded disk  $\Delta_1$
in $S\setminus \tilde h_S(\Sigma\times D^2)$, namely $\Delta_1=\mathfrak{E}_1\cap \big(S\setminus \tilde h_S(\Sigma\times D^2)\big)$.   In particular, the inclusion $\pi_1\big(S  \setminus  \tilde\Sigma_{S,\delta} ,e\big)\to \pi_1(S,e)$
 is an  isomorphism. 

\subsection{ The  building block $(S, \tilde h_S)$}

We finish this section with a summary of the properties of the building block $(S, \tilde h_S)$ which are established above.
 \begin{thm}\label{thmS}The pair  $(S, \tilde h_S\colon \Sigma\times D^2\to S)$ satisfies:
\begin{enumerate}
\item $
\pi_1\big(S, e\big)=\langle  x,y,z, t\mid [y,\bar x],~{\rm Rel} \rangle
$ and $H_1(S)\cong\ZZ^2$, generated by $y,z$
\item $\tilde h_S(\Sigma\times\{1\})=\tilde\Sigma_{S}=E_0\cup \tilde A\cup T_{z,0}$. 
\item $\tilde h_S(\Sigma\times\{0\})=\tilde\Sigma_{S,\delta}$. The subsurface $E_{0,\delta}\subset  \tilde \Sigma_{S,\delta}$ equals $E_0\times \{(\tfrac{\delta}{\sqrt{2}},\tfrac{\delta}{\sqrt{2}})\}$.
\item The base point $e$ and the loops $x,y,\mu_z$ and $\mu_{S,\delta}$  lie  on $\tilde h_S(\Sigma\times S^1)$, hence in $S\setminus \tilde\Sigma_{S,\delta}$.

 \item
The loops $ z$ and $t$ lie in  $S\setminus \tilde\Sigma_{S,\delta}$. 

 \item The meridian $\mu_{S,\delta}=\tilde h_S(\sigma\times S^1)$ of $\tilde\Sigma_{S,\delta}$ bounds a smoothly embedded 2-disk $\Delta_1$ in $S \setminus \Nbd{\tilde\Sigma_{S,\delta}}$. 
\item The inclusion $\pi_1\big(S  \setminus  \tilde\Sigma_{S,\delta} ,e\big)\to \pi_1(S,e)$
 is an  isomorphism. 
 \item The composite 
 $$ \pi_1(\Sigma\times\{1\},\sigma)\xrightarrow{\tilde h_S} 
 \pi_1(\tilde\Sigma_S,e)
\xrightarrow{i_S} \pi_1(S,e)$$ is given by
$$ a_1\mapsto z, ~ b_1\mapsto t, ~ a_2\mapsto   y, ~ b_2\mapsto \bar x.$$
\end{enumerate}\qed
\end{thm}

\section{The surface sum $Z=S\#_\psi P$ along $\tilde\Sigma_S$ and $\Sigma_P$}\label{sec7}

  The relation
$[a_1,b_1][a_2,b_2]=1$  in $\pi_1(\Sigma,\sigma)$ can be rewritten as
$[a_1 b_1 \bar a_1, \bar a_1][a_2 b_2 \bar a_2, \bar a_2]=1.$
Moreover $(a_1 b_1 \bar a_1)\cdot (\bar a_1)=1=(a_2 b_2 \bar a_2)\cdot (\bar a_2)$, and the four loops $a_1 b_1 \bar a_1, \bar a_1 , a_2 b_2 \bar a_2, \bar a_2$ form two disjoint embedded hyperbolic pairs.

  Lemma \ref{lem3.1}  implies that  there exists a diffeomorphism $\Psi:(\Sigma,\sigma)\to(\Sigma,\sigma)$ inducing
\begin{equation}\label{glue2}a_1\mapsto a_1b_1\bar a_1, b_1\mapsto \bar a_1,~ a_2\mapsto a_2 b_2 \bar a_2,~ b_2\mapsto \bar a_2\end{equation}
on $\pi_1(\Sigma,\sigma)$.\footnote{Any diffeomorphism of $\Sigma$ can be used to  produce a symplectic manifold, but $\Psi$ suffices for our constructions.}
\begin{df}\label{defZ}
Define  the {\em (symplectic) surface sum} of $S$   and $P$ along their marked  (symplectic) surfaces $\tilde\Sigma_S$ and $\Sigma_P$  to be
\begin{align*}  Z&:=\big(S\setminus \tilde h_S(\Sigma \times D^2)\big)\cup_{\psi}\big(P\setminus \tilde h_P(\Sigma \times D^2)\big),\end{align*}
with gluing diffeomorphism $\psi$  given by
the composite:
$$\psi:\partial \Nbd{\tilde\Sigma_{S,\delta}}\xrightarrow{\tilde h_S^{-1}}\Sigma\times S^1\xrightarrow{\Psi\times -{\rm Id}}
\Sigma\times S^1
\xrightarrow{\tilde h_P}\partial \Nbd{\Sigma_{P,\delta}}.
$$
See Figure \ref{fig57fig} for a schematic. \end{df}
 
 We write $e$ for the base point  of $Z$, with the understanding that   the loops $x,y,z,t$ and $\mu_{S,\delta}$ in $S$ are based at $e$ and $x',y',z',t',x'',y'',z'',t''$ and the meridian $\mu_{P,\delta}$ in $P$ are based at $e_z'=e_z''$.

 The isomorphism $\psi_\#\colon\pi_1(\partial\Nbd{\tilde\Sigma_{S,\delta}})\to \pi_1(\partial\Nbd{\Sigma_{P,\delta}})$ induced by $\psi$ is given by $$
\zeta\mapsto y'\bar x' \bar y', ~ \tau\mapsto \bar y', ~y\mapsto y''\bar x'' \bar y'', ~ \bar x\mapsto \bar y'',~\mu_{S,\delta}\mapsto \mu_{P,\delta}^{-1}.
$$

\subsection{Fundamental group of $Z$}

Set  \begin{equation}\label{glue}
{\rm Sum} 
=\{\bar z(y'\bar x' \bar y'),~ \bar t\bar y',~ \bar y(y''\bar x''\bar y''), ~
 \bar  x y''\}.
\end{equation}

\begin{thm} \label{thm5.1}The fundamental group  $\pi_1(Z,e) $ has a presentation with $12$ generators:
$$
x,y,z,t,x',y',z',t',x'',y'',z'',t''
$$
a complete set of 32 relations is given by
$$
[y,\bar x]=1,~{\rm Rel}, ~{\rm Rel}', ~{\rm Rel}'', ~{\rm Gl} , ~ {\rm Sum}.
$$ 
Hence $Z$ is simply connected.
\end{thm}

 \begin{wrapfigure}[11]{r}{0.4\textwidth}
\begin{center}
\vskip-.3in
\includegraphics[width=2.7in]{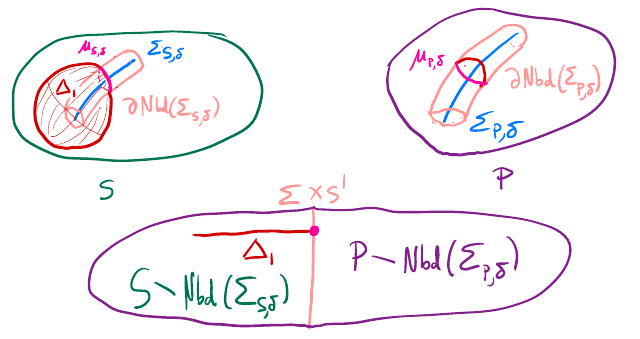}
 \vskip-.1in
\caption{\label{fig57fig} $S\#_\Sigma P= S^\#\cup P^\#$.}
\end{center}
\end{wrapfigure}
\noindent{\em Proof.} To apply the  SVK theorem efficiently, we decompose $Z$ slightly differently than (\ref{defZ}). 
Set $$S^\#:=S\setminus \Nbd{\tilde\Sigma_{S,\delta}}\hskip1.5in$$ and 
$$ 
P^\#:= \big(P \setminus 
 \Nbd{\Sigma_{P,\delta}}\big)\cup \Nbd{\Delta_1},\hskip1.5in
$$
where $\Delta_1\subset S^\#$ is the embedded 2-disc  whose boundary circle is    $\mu_{S,\delta}=\tilde{h}_S(\sigma\times S^1)$ of Theorem \ref{thmS} (6).  
Then
$$Z = S^\#\cup P^\#.$$  See Figure \ref{fig57fig}.
The overlap $S^\#\cap P^\#$ is 
$$\partial\Nbd{\tilde\Sigma_{S,\delta} }\cup \Nbd{\Delta_1}=(\tilde\Sigma_{S,\delta}\times S^1)\cup_{\mu_{S,\delta}} \Nbd{\Delta_1},$$ and therefore $\pi_1(S^\#\cap P^\#,e)=\pi_1(\tilde\Sigma_{S,\delta}\times S^1,e)/\langle\langle\mu_{S,\delta}\rangle\rangle=\pi_1(\tilde\Sigma_{S,\delta},e)$.
The composite $$\pi_1(\Sigma,\sigma)\xrightarrow{\tilde h_S(-,1)} \pi_1(\partial\Nbd{\tilde\Sigma_{S,\delta}},e)\to \pi_1(S^\#\cap P^\#,e)$$ is an isomorphism.
The inclusion $\pi_1(S^\#,e)\to \pi_1(S,e)$ is an isomorphism by Theorem \ref{thmS}. 

The two surjections $\pi_1(P\setminus\Nbd{\Sigma_{P,\delta}})\to \pi_1(P)$ and $\pi_1(P\setminus\Nbd{\Sigma_{P,\delta}})\to \pi_1(P^\#)$ have the same kernel and hence the zig-zag induces an isomorphism $\pi_1(P,e)\cong \pi_1(P^\#,e)$.  It follows that the SVK diagram
\[\begin{tikzcd}
\pi_1(S^\#\cap P^\#,e)\arrow[r]\arrow[d]&\pi_1(S^\#,e)\arrow[d]\\
\pi_1(P^\#,e)\arrow[r]&\pi_1(Z,e)
\end{tikzcd}\]
can be rewritten as \[\begin{tikzcd}
\pi_1(\Sigma,\sigma)\arrow[r]\arrow[d]&\pi_1(S,e)\arrow[d]\\
\pi_1(P,e)\arrow[r]&\pi_1(Z,e)
\end{tikzcd}\]
which  the calculations above identify as
\[\begin{tikzcd}
\langle a_1,b_1,a_2,b_2\mid [a_1,b_1][a_2,b_2]\rangle\arrow[r]\arrow[d]&\langle x,y,z,t\mid [y,\bar x],~{\rm Rel}\rangle\arrow[d]\\
 \langle z', x' ,y',t' ,
 z'', x'', y'',t''
 \mid  {\rm Rel}',~   {\rm Rel}'',~{\rm Gl}  \rangle\arrow[r]&\pi_1(Z,e)
\end{tikzcd}\]
The maps out of $\pi_1(\Sigma,\sigma)$ are calculated in Theorems \ref{thmS} and \ref{thmP}. 
The first assertion now follows from the SVK theorem. 

\medskip

It remains to prove that $Z$ is simply connected. For the reader's convenience, we recall that the four relations in Sum  read: 
$$
 z= y'\bar x' \bar y' ,~  t=\bar y',~  y=y''\bar x''\bar y'', ~
 x =y'';
$$
the three relations in  Gl read:
$$
 ~  z'=\bar t'', ~t'= z'',  ~[y',\bar x'][y'',\bar x'']=1;
$$
and the six surgery relations  Sur,  Sur$'$, Sur$''$ read:
\begin{equation}\label{six}
x=[z,\bar y], ~ t =[x, \bar z], 
x'=[z',\bar y'], ~  t'=[x', \bar z'],
x'' =[z'',\bar y''], ~  t''=[x'', \bar z''].
\end{equation}

\begin{lem}\label{trick2}
Using (only) the relations 
\begin{enumerate}
\item $  z'=\bar t''$   from Gl,    $[t'', y'']=1, [t'',x'']=1$ from tCen$''$, 
and $x'=[z',\bar y']$ from Sur$'$ in $\pi_1(P)$ (which hold in $\pi_1(P\setminus \Sigma_{P,\delta})$ by Lemma \ref{Pcomp}),
\item $[y,t]=1$ from tCen in $\pi_1(S)$
\item   $z= y'\bar x' \bar y' , ~  t=\bar y' $  and $y=y''\bar x''\bar y''$  from Sum,
\end{enumerate}
 it follows that $[z, \bar y]=1\text{ in } \pi_1(Z).$
\end{lem}
\begin{proof}   Using  $z= y'\bar x' \bar y' $  in Sum and $x'=[z',\bar y']$ in Sur$'$:
\begin{equation}\label{magic}[z,\bar y]=[y'\bar x' \bar y' ,\bar y]
=[y' [\bar y',z'] \bar y' ,\bar y]=
[[\bar y', y'z'\bar y'], \bar y].\end{equation}  
Using  $  t=\bar y' $ in Sum and $[y,t]=1$ in tCen,  it follows that
 $[ \bar y', \bar y]=[t, \bar y]=1$.  Using  $  z'=\bar t''$ in Gl,   $y=y''\bar x''\bar y''$  in Sum,  and $[t'', y'']=1, [t'',x'']=1$ in  tCen$''$,   it follows that
 $$[y'z'\bar y', \bar y]=y'[z',\bar y]\bar y'=y'[  \bar t'',   y''x''\bar y'']\bar y'
=1.$$
Hence $\bar y$ commutes with $\bar y '$ and with $y'z'\bar y'$, so that $y$ commutes with $[\bar y', y'z'\bar y']$.   Now (\ref{magic}) implies that $[z,\bar y]=1$.\footnote{S. Baldridge showed me this  trick in January 2007.}
\end{proof}
Returning to the proof that $Z$ is simply connected: using the first surgery relation, Lemma \ref{trick2} implies $x=1$ in $\pi_1(Z)$.   Using Sum, this proves $y''=1$. 
Since $x=1$, the second surgery relation proves $t=1$. Since $y''=1$,  the fifth surgery relation shows  $x''=1$.   
Since $x''=1$, the sixth surgery relation shows $t''=1$, which, using Gl, implies $z'=1$. Also, since $x''=1$, Sum shows that $y=1$.

Since $z'=1$, the third and fourth surgery relations  imply that $x'=1$ and $t'=1$.  Since $t'=1$, Gl implies that $z''=1$.

Since $t=1$, Sum implies that $y'=1$. Since $x'=1$, Sum implies that $z=1$,   so that all twelve generators  are trivial, and therefore $Z $ is simply connected.
\qed
 
\medskip

  From the construction, it is straightforward to calculate that
the Euler characteristic of $Z$ equals 6, and Novikov additivity implies that the signature of $Z$ equals $-2$. Since $\pi_1(Z)=1$, $H_2(Z)$ has rank 4. The classification of indefinite unimodular forms shows that the intersection form of $Z$ is isometric to $(1)\oplus (-1)^3$. Since $Z$ is smooth, it has vanishing Kirby-Siebenmann invariant. Freedman's theorem \cite{Freedman} shows that $Z$ is homeomorphic to $\CP ^2\# 3\overline{\CP }^2$.  
See Appendix \ref{exotic} for a proof that $Z$ is not diffeomorphic to $\CP ^2\# 3\overline{\CP }^2$.

\section{The cancellation  trick,  and Theorem 2 of \cite{BK} }

Two topics are discussed in this section. First,   we formalize the cancellation trick of Lemma \ref{trick2}.   Second, we provide a quick proof
of Theorem 2 of \cite{BK} as a consequence of Proposition \ref{thm1}.

\subsection{Attaching a 4-manifold to $P\setminus \Nbd{\Sigma_{P,\delta}}$}
We isolate the cancellation described in the proof of Lemma \ref{trick2}, which lies at the heart of the proof of Theorem \ref{thm5.1},  in the following proposition, which replaces 
$S\setminus \Nbd{\tilde \Sigma_{S,\delta}}$ by an arbitrary 4-manifold $W$ with boundary $\Sigma\times S^1$.

 The statement is technical in order to be precise, but 
loosely, it asserts that if $b_1$ and $a_2$ in $\pi_1(\partial W)=\pi_1(\Sigma)\times\ZZ$ commute in $\pi_1(W)$, then gluing  $W$ to  $P\setminus \Nbd{\Sigma_P}$ using the diffeomorphism  $\psi$ of Definition \ref{defZ} yields a closed 4-manifold $Z_W$ so that   $a_1$ and $a_2$ commute in $\pi_1(Z_W)$. 

\begin{thm}\label{trick}
Suppose $W$ is a 4-manifold, and $\tilde h_W:\Sigma\times S^1\to \partial W$ a diffeomorphism.  Suppose further that $$[\tilde h_W(b_1, 1), \tilde h_W(a_2 , 1)]=1\text{ in }\pi_1(W).$$
Denote by $Z_W$  the  closed 4-manifold
$$ Z_W=W\cup_{ \tilde h_P\circ (\Psi\times -{\rm Id})\circ \tilde h_W^{-1}   }(P\setminus \Nbd{\Sigma_{P,\delta}}).
$$

Then $$
[ \tilde  h_W(a_1,1),
 \tilde  h_W(a_2,1) ]=1= [y'\bar x' \bar y', y'' \bar x'' \bar y''] \text{ in } \pi_1(Z_W).$$  
 \end{thm}
 
\begin{proof}  Lemma \ref{trick2}  uses $ z'=\bar t''$   from Gl,    $[t'', y'']=1, [t'',x'']=1$ from tCen$''$, 
and $x'=[z',\bar y']$ from Sur$'$ in $\pi_1(P)$ .  These relations hold in $\pi_1(P\setminus \Sigma_{P,\delta})$ by Lemma \ref{Pcomp}. 

  Replace $y=\tilde h_S(a_2,1),z=\tilde h_S(a_1,1),$ and $t=\tilde h_S(b_1,1)$ by $\tilde h_W(a_2,1), \tilde h_W(a_1,1), $ and $\tilde h_W(b_1,1)$ in the proof of Lemma \ref{trick2}.   Note that  hypothesis (2) is the only requirement from $S$ needed for the conclusion. Hence
$$1= [\tilde h_W(a_1,1),
\tilde h_W(a_2,1)]=[\tilde h_P\circ \Psi(a_1) ,\tilde h_P\circ \Psi(a_2) ] =[y'\bar x' \bar y', y''\bar x'' \bar y''].$$
\end{proof}

\subsection{Theorem 2 of \cite{BK}} We next state and prove  Theorem \ref{lemDumb},  a strengthening  of the main technical result of \cite{BK}. The results of this section are not used anywhere else in this article. This section can be read immediately after   Corollary \ref{cor2}.  \medskip

In this section only, set $E'=\TT^2\times \{(-\ep,0)\}\subset \TT^4$ and
$E_\delta'=\TT^2\times \{(-\ep,\delta)\}\subset \TT^4$.  
 
\begin{thm} \label{lemDumb} The fundamental group
 of the complement of  $T_x,T_y$ in $\TT^4\setminus \Nbd{E_\delta'\cup T_z}$  has the presentation
 $$
\langle x,y,  z, t\mid {\rm Bor}, [t,x], [t, y], [t,z\bar y\bar z], [t, z x\bar z]\rangle.$$
Moreover, in this group $\mu_x=[  z,\bar y]$ and  $\mu_y=[x,   \bar z]$. 
 The fundamental groups of the   3-tori $\partial \Nbd{T_x}$ and  $\partial \Nbd{T_y}$ are generated by $\{x,t,\mu_x\}$ and $\{y\bar \mu_x, t,\mu_y'=[zx\bar z,\bar z]\}$.
\end{thm}
\begin{wrapfigure}[12]{r}{0.25\textwidth} 
\begin{center}
\vskip-.5in
\includegraphics[width=1.8in]{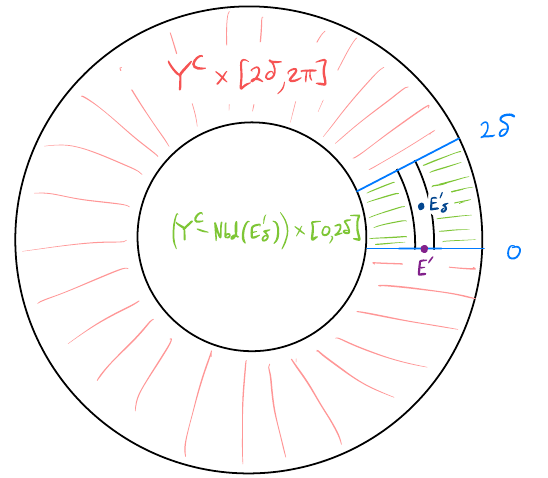}
 \vskip-.1in
\caption{\label{fig65fig} \color{red}}
\end{center}
\end{wrapfigure}
\noindent{\em Proof.}   The manifold   
$$W:=(\TT^4\setminus \Nbd{E_\delta'\cup T_z})\setminus \Nbd{T_x\sqcup T_y } =(Y^c\times S^1)\setminus \Nbd{E_\delta'}\hskip1.4in $$ 
is homeomorphic to (see Figure \ref{fig65fig}):
$$
Y^c\times[2\delta,2\pi]\cup_{(Y^c\setminus\Nbd{E'})\times\{0,2\delta\})}(Y^c\setminus\Nbd{E'})\times[0,2\delta].\hskip1in$$
Hence, using Proposition \ref{thm1} and the SVK theorem,  its fundamental group has a description of the form
$$\pi_1(W,e)=\big(\pi_1(Y^c)*\langle t\rangle\big)
/N=\langle x,y,  z, t\mid {\rm Bor} \rangle/N$$
where $N$ is the normal subgroup generated by 
$\{[\gamma,t] \mid  \gamma\in \pi_1 \big(\TT^3 \setminus \Nbd{C_x\sqcup C_y\sqcup (C_z\cup E'})\big) \}$ (notice that $\TT^3 \setminus \Nbd{C_x\sqcup C_y\sqcup (C_z\cup E'})=Y^c\setminus\Nbd{E'}$).

\bigskip

\begin{wrapfigure}[14]{r}{0.26\textwidth}
\begin{center}
\vskip-.3in
\includegraphics[width=2.3in]{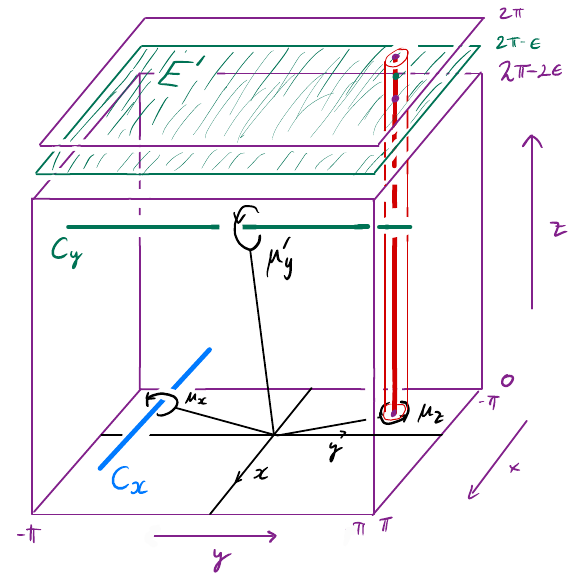}
\vskip-.2in
\caption{\label{fig55fig}}
\end{center}
\end{wrapfigure}
To prove the first statement,  it suffices to show that $x,y, [z,\bar y ]$, and $z[ x,\bar z]\bar z$ generate $\pi_1\big(\TT^3 \setminus \Nbd{C_x\sqcup C_y\sqcup (C_z\cup E')}\big)$, where $E'=\TT^2\times\{-\ep\}\subset \TT^3$. This can be seen using Figure \ref{fig55fig}, which shows $E',C$ inside the fundamental domain $[-\pi,\pi]^2\times[0,2\pi]$ for $\TT^3$. Also shown is a meridian $\mu_y'$ for $C_y$ (the meridian $\mu_y$ does not lie in $\TT^3\setminus\Nbd{E'}$).

Collapsing the $z$ coordinate  provides a deformation retract
of $\TT^3\setminus\Nbd{ E' \cup C}$ to a wedge of two 2-tori:
the first has fundamental group generators $x$ and $\mu_x$, the second has generators $y\bar\mu_x$ and  $\mu_y'$. 
The inclusion  
$\TT^3 \setminus \Nbd{C\cup E'}\subset Y^c$ takes $\mu_x$ to $[z,\bar y]$, $\mu_z$ to $[y,\bar x]$ and 
 $\mu_y'\text{ to } z\mu_y\bar z=z[x,\bar z]\bar z=[zx\bar z,\bar z].$    Hence 
$$\pi_1(W,e)=\langle x,y,  z, t\mid {\rm Bor}, [t,x], [t, y], [t, [z,\bar y]], [t, z[x,\bar z]\bar z]\rangle.\hskip1.5in$$

The pair of relations $[t,y]=1, [t,[z,\bar y]]=1$ and the pair $[t,y]=1,[t,z\bar y\bar z]=1$ span the same normal subgroup, hence we may replace $[t,[z,\bar y]]=1$ by $[t,z\bar y\bar z]=1$.
Similarly $[t,x]$ and $[t, z[x,\bar z]\bar z]$ span the same normal subgroup as $[t,x]$ and $[t, z x\bar z]$.  This proves the first statement.

From Figure \ref{fig55fig} one sees that the loops $x,t,\mu_x$ generate the fundamental group of $\partial \Nbd{T_x}$, and $y\bar \mu_x, t,\mu_y'=[zx\bar z,\bar z]$ generate the fundamental group of $\partial \Nbd{T_y}$.
\qed

\medskip
\begin{wrapfigure}[10]{r}{0.4\textwidth}
\begin{center}
\vskip-.3in
\includegraphics[width=3.3in]{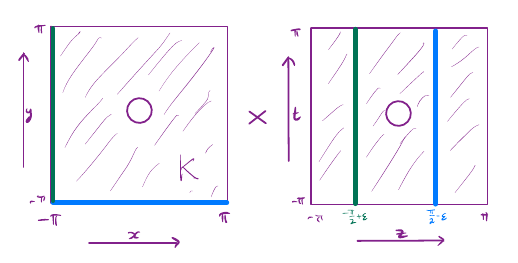}
\vskip-.2in
\caption{\label{fig56fig} \color{red}}
\end{center}
\end{wrapfigure}
\subsubsection{Comparison to \cite[Theorem 2]{BK}}
 Define  the diffeomorphism 
 $\eta:\TT^4\to \TT^4=(\RR/2\pi\ZZ)^4$ by $$\eta(a,b,c,d):= (a-\tfrac\pi 2, b+\tfrac\pi 2, c  -  \ep, d+\delta),\hskip2.5in $$
with inverse $
\eta^{-1}(x,y,z,t)=(x+\tfrac\pi 2,y-\tfrac\pi 2,z+\ep,t-\delta).$
Let $K:= \{(x,y)\in \TT^2\mid {\rm dist}((x,y),(0,0))\ge\ep\}$, a punctured torus.

  Note that $\eta(\TT^2\times \{(0,0)\})= E_\delta'$ and $\eta(  \{(0,0)\}\times \TT^2)= T_z$. Therefore, the     restriction $\eta:K\times K \to  \TT^4$ is an embedding with image $\TT^4\setminus\Nbd{E_\delta'\cup T_z}$. Moreover
  $$
  \eta^{-1}(T_x)=\{(a, -\pi, \tfrac\pi 2 +  \ep,d)\mid a,d\in S^1\}  \text{ and }
  \eta^{-1}(T_y)=\{(\pi,b, - \tfrac\pi 2 +  \ep,d)\mid b,d\in S^1\}.
$$
   These are illustrated in Figure \ref{fig56fig}.
 Hence Theorem \ref{lemDumb} provides a presentation for the fundamental group of the complement of   $ \eta^{-1}(T_x)\cup \eta^{-1}(T_y)$ in $K\times K$, a product of 
 two punctured tori.  
 
Theorem 2 of \cite{BK} provides a set of generators and some relations for the fundamental group of 
the complement of a pair $T_1,T_2$ of coordinate tori in the product of two punctured tori $K\times K$ (the second factor is denoted $H$ in \cite{BK}).  It asserts that 
$\pi_1(K\times K\setminus (T_1\sqcup T_2))$ is generated
by four loops $X,Y,A,B$ (these are denoted by the lower case letters $x,y,a,b$ in \cite{BK}; we use upper case here to avoid conflict). Moreover, it asserts that the relations
 $
[X,A],[Y,A],~ [Y, BA\bar B],~ [[X,Y],B],$ and $ [X,[A,B]]
$ 
hold in $\pi_1(K\times K\setminus (T_1\sqcup T_2))$.
(It also asserts that the additional relation $[Y,[A,B]] $ holds
in $\pi_1(K\times K\setminus (T_1\sqcup T_2))$; the reader can verify easily that this is a consequence of $[Y,A]$ and $[Y, BA\bar B]$.)

Using $[X,A]=1$, one can replace $[X,[A,B]]=1$ by $[A,\bar B X B]=1$. One can also replace $[Y, B A \bar B]=1$ by $[\bar B Y B,   A  ]=1$.
Hence Theorem 2 of \cite{BK} can be restated as: $\pi_1(K\times K\setminus (T_1\sqcup T_2))$ is a quotient of 
$$\langle X,Y,A,B\mid [X,A], [Y,A],  [A,\bar  B Y B], [A,\bar B X B],[[X,Y],B]\rangle.$$
The theorem also asserts that the based meridians $\mu_1,\mu_2$ of the tori $T_1, T_2$ are given by  $\mu_1=[\bar B,\bar Y]$   and $\mu_2=[\bar X, B]$, and that the peripheral fundamental group   $\pi_1(\partial \Nbd{T_1})$ is generated by $\mu_1, X,$ and $A$, and    $\pi_1(\partial \Nbd{T_2})$ is generated by $\mu_2, Y,$ and $BA\bar B$.

\medskip

It is easy to check that the substitutions 
$$
X\mapsto \bar x,~ Y\mapsto y,~ A\mapsto t, B\mapsto \bar z$$
define  a surjective homomorphism to $
\langle x,y,  z, t\mid {\rm Bor}, [t,x], [t, y], [t,z\bar y\bar z], [t, z x\bar z]\rangle $
which sends the triple $ \{\mu_1, X,A\}$   to $\{ \mu_x, \bar x, t\}$, and the triple $\{\mu_2, Y, BA\bar B\}$ to $\{\mu_y, y, \bar z t z\}$. 
Conjugating the triple $\{\mu_y, y, \bar z t z\}$ by $z$ yields
$$\{z[x,\bar z]\bar z, zy\bar z, t\}= \{[zx\bar z, \bar z], z y\bar z, t\}= \{\mu_y',   y\bar \mu_x, t\}
$$
Hence Theorem \ref{lemDumb} implies Theorem 2 of \cite{BK}; the paths from the base point to the  torus $T_y=\eta(T_2)$
in the two articles differ by precomposition with $z$.

\section{A simply connected homology $\CP ^2 \#2\overline {\CP }^2$ }  \label{DA}
Akhmedov-Park \cite{AP2}  introduce another pair of building blocks which they label 
 $Y'(1)$ and $Z''(1,1)$, each containing embedded genus two surfaces of square zero, and prove that the  surface sum $Y'(1)\#_\Sigma Z''(1,1)$ is a simply connected (symplectic, irreducible, hence exotic) homology $\CP ^2
\#2\overline {\CP }^2$.  
In this section we produce a similar example using the manifolds $\TT^4_\sur$ and $P$ constructed above. 
The   genus two surface $\tilde \Sigma_{AP}$ in $\TT^4\#\overline{\CP }^2$ described below is identical with the one introduced in \cite{AP2}.

\subsection{A  torus $E_{AP}$ in $\TT^4_\sur$ satisfying $[E_{AP}]=2[E]$.}\label{em}

 \begin{wrapfigure}[11]{r}{0.33\textwidth}
\begin{center}
\vskip-.4in
\includegraphics[width=2.6in]{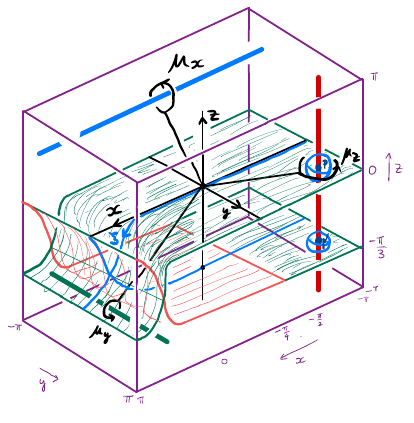}
\vskip-.2in
\caption{\label{fig44fig}  The torus $E_{AP}$\color{red}}
\end{center}
\end{wrapfigure}
We begin with the construction of a (symplectic) torus $E_{AP}$ in $\TT^4$ which intersects $T_z$ positively in two points.

\medskip

Figure \ref{fig44fig} shows an {\em embedded} (symplectic) torus $E_{AP}$ in $\TT^2\times [-\tfrac\pi 3,0]\times[0, \tfrac\ep 4]$ which represents $2[E]$ in $H_2(\TT^4)$. The annulus shaded orange is moved slightly into $\{t>0\}$. Also shown are the curves $x,y,z,\mu_x,\mu_y,\mu_z$ defined in Section \ref{sect2} and Figure \ref{fig26fig}. The intersection of $E_{AP}$ and $T_z$ consists of the point $p$ and another point $p'$.
A loop $\xi$ is also shown; it lies on $E_{AP}$, as do $y$ and $\mu_z$. The loops $z,t,\mu_x,$ and $\mu_y$ do not lie on $E_{AP}$, they intersect $E_{AP}$ only at the point $e$. Only part of the loop $x$ lies on $E_{AP}$.

\vskip.2in

  \begin{wrapfigure}[5]{r}{0.35\textwidth}
\begin{center}
\includegraphics[width=2.4in]{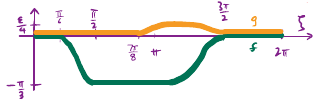}
\vskip-.25in
\caption{\label{fig42fig} $f$ and $g$}
\end{center}
\end{wrapfigure}
A precise construction of $E_{AP}$ is: let $g:[0,2\pi]\to [0, \tfrac\ep 4]$ be a smooth function which is positive on $(\tfrac{7\pi} 8, \tfrac{3\pi}2)$, equals $\tfrac\ep 4$ on $[\tfrac{15\pi} {16}, \tfrac{23\pi}{16}]$  and zero elsewhere.
  Let $f:[0, 2\pi]\to [-\tfrac\pi 3,0]$ be a smooth function which vanishes on $[0, \tfrac\pi 6]\cup [\tfrac{3\pi}2,2\pi]$ and equals $-\tfrac\pi 3$ on
 $[\tfrac\pi2, \tfrac{7\pi} 6]$. These are graphed in Figure \ref{fig42fig}.
 Define a smooth function $m:S^1\times S^1\to \TT^4$ by
$$m(\xi,y):= (2\xi, y, f(\xi), g(\xi)).\hskip1.5in$$  The function $m$ embeds $S^1\times S^1$ into $\TT^4$, missing $T_x$ and $T_y$.  
 Define $E_{AP}$ to be the image of the embedding $m$, and use $m$ to define the (embedded) loop $$\xi:S^1\to E_{AP}, ~\xi(\theta)=m(\theta,0),~ 0\leq\theta\leq 2\pi.$$ 
Clearly, $\pi_1(E_{AP},e)=\langle \xi, y\mid [\xi,y]\rangle$, and from Figure \ref{fig44fig} one sees that 
$\xi$ is sent to $x^2$ in $\pi_1(\TT^4_\sur,e)$.
 The torus $E_{AP}$ intersects $T_z$ transversely and positively in two points: the point 
$p=(-\tfrac\pi 2,\tfrac\pi 2, 0,0)$ and a second point $p':=(-\tfrac\pi 2,\tfrac\pi 2, - \tfrac\pi 3,0)$. Therefore $E_{AP}\cdot T_z=2$.

\subsection{A  genus two surface $\tilde\Sigma_{AP}$ in $\TT^4_\sur\#\overline{\CP }^2$}
 
 Resolving $E_{AP}\cup T_z$ at $p$ results in an immersion of a genus two surface 
 $\iota:\Sigma_{AP}=(E_{AP} \setminus \Nbd{p})\cup A\cup T_{z,0}\to \TT^4_\sur.$
  More precisely: form an (abstract) genus two surface  
  $$\Sigma_{AP}:=(E_{AP} \setminus \Nbd{p})\sqcup A\sqcup T_{z,0}/\sim,$$
where  $A$ is  the same annulus in a neighborhood of $p$ as used  in the construction of $\Sigma_S$ (\ref{eq3.6}).  Points on the  boundary circles are identified if they are equal in $\TT_\sur^4$. The two points, one in $E_{AP}$ and one in $T_{z,0}$, mapping $p'$ are {\em not} identified in $\Sigma_{AP}$, and the immersion $\iota:\Sigma_{AP}\to \TT^4_\sur$ is the inclusion, with the single double point $p'$. These are illustrated in Figure \ref{fig40fig}.
The $\ep$-neighborhoods of $p'$ in $E_{AP}$ and $T_z$ are also shown.

  \begin{wrapfigure}[8]{r}{0.55\textwidth}
 \begin{center}
\includegraphics[width=4in]{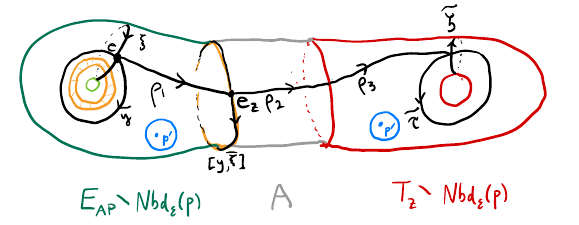}
 \vskip-.1in
\caption{\label{fig40fig} The surface $\Sigma_{AP}$. \color{red} }
\end{center}
\end{wrapfigure}
 The immersion $\iota$ has a single, transverse,  self-intersection point, at $p'$.   Moreover $\Sigma_{AP}\cdot \Sigma_{AP}=(E_{AP}+T_z)\cdot(E_{AP}+T_z)=  E_{AP}^2 +2E_{AP}\cdot T_z+T_z^2=4$.

  The  based loops $\zeta=(\rho_1\rho_2\rho_3)\tilde \zeta (\rho_1\rho_2\rho_3)^{-1}$ and $\tau=(\rho_1\rho_2\rho_3)\tilde \tau (\rho_1\rho_2\rho_3)^{-1}$ (see (\ref{tildezetatau}) and  Proposition \ref{lem1.33}), as well as the loops $\xi, y$ and $\mu_z$, lie on $\Sigma_{AP}$. 
   The  loops $\xi, y, \zeta,\tau$ in $\Sigma_{AP}$ based at $e$ satisfy the surface relation $$
 [y,\bar\xi][\zeta,\tau]=1\hskip2.5in$$ in $\pi_1(\Sigma_{AP}),$ and the map $\pi_1(\Sigma_{AP},e)\to \pi_1(\TT^4_\sur,e)$ takes $\xi$ to $x^2$.  Moreover $[\zeta,\tau]= \bar\mu_z$ in $\pi_1(\Sigma_{AP},e)$.  See Figures \ref{fig44fig} and \ref{fig40fig}.

\medskip
  
  Blow up $\TT^4_\sur$ once, at $p'$, and denote the result  $$R:=\TT^4_\sur\#\overline{\CP }^2.$$
 The blow-down map $R\to \TT^4_\sur$, collapsing the exceptional 2-sphere to a point,  induces an isomorphism on fundamental groups. Let $\tilde \Sigma_{AP}\subset R$ denote  the proper transform of $\Sigma_{AP}$ (see Appendix \ref{CP}).    The immersion $\iota:\Sigma_{AP}\to \TT^4_\sur$ lifts to an embedding $\iota:\Sigma_{AP}\to R$ with image $\tilde \Sigma_{AP}$.  
  Proposition \ref{prop2.1} implies that
 $$\pi_1(R,e)=\langle x,y,z,t\mid [y,\bar x], {\rm Rel}\rangle.$$
Since the loops $\xi,y,\zeta,$ and $\tau$ avoid the blow up point $p'$, they lie on $\tilde\Sigma_{AP}$ and satisfy
 $
 [y,\bar\xi][\zeta,\tau]=1$ in $\pi_1(\tilde\Sigma_{AP},e)$.

 \medskip
 
 The surface  $\tilde \Sigma_{AP}$ satisfies $\tilde \Sigma_{AP}\cdot \tilde \Sigma_{AP}=0$ (see Proposition 2.3.5 of \cite{GS}). In particular, $\tilde \Sigma_{AP}$ admits a unit normal field.

 \subsection{Framing $\tilde \Sigma_{AP}$}  
 \begin{wrapfigure}[10]{r}{0.6\textwidth}
\centering
   \vskip-.3in
\includegraphics[width=4.3in]{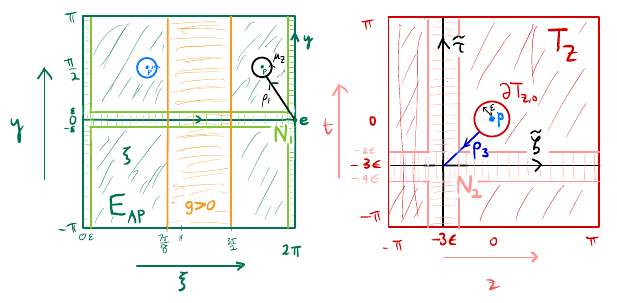}
 \vskip-.2in
\caption{\label{fig45fig} $\xi\cup y\subset N_1\subset E_{AP}$ and $\tilde \zeta \cup \tilde \tau\subset N_2\subset T_z$. \color{red}}
\end{wrapfigure}
\begin{df} Let 
$N_1$ denote the $\ep$-neighborhood of $\xi \cup y$  in $  \tilde \Sigma_{AP}$, and let 
$ N_2$  denote an $\ep$-neighborhood of  $\tilde \tau\cup\tilde \zeta $ in $  \tilde \Sigma_{AP}$. The $N_i$ are punctured tori with four corners on their boundary.

 Notice that  $N_1\subset  E_{AP}\setminus \Nbd{p\sqcup p'}$ and $N_2\subset T_z\setminus \Nbd{p\sqcup p'}$. These neighborhoods are illustrated in Figure \ref{fig45fig}.\end{df}

\begin{lem}\label{framing} 
There exists a transverse vector field $\nu$ on $\tilde\Sigma_{AP}$: 
\begin{enumerate}
\item the restriction of  $\nu$ to $N_1$ equals $ \tfrac{\partial}{\partial z} +\tfrac{\partial}{\partial t}$,
\item the restriction of $\nu$ to $N_2$ equals $-\tfrac{\partial}{\partial x}$.
\end{enumerate}
This choice of $\nu$ ensures that for small enough $\delta>0$,
\begin{enumerate}
\item[(3)] the $\delta$-push-off $N_{1,\delta}$ equals
$N_1+ (0,0,\tfrac{\delta}{\sqrt{2}},\tfrac{\delta}{\sqrt{2}})$,
\item[(4)] the $\delta$-push-off $N_{2,\delta}$  equals $N_2+ (-\delta,0,0,0)$.
\end{enumerate}

\end{lem}
\begin{proof}  The restriction 
$$H^1(\tilde \Sigma_{AP })=[\tilde \Sigma_{AP },S^1]\to H^1(N_1\sqcup N_2)=[N_1\sqcup N_2,S^1]$$ is an isomorphism, by the long exact cohomology sequence for the pair $(\tilde \Sigma_{AP },N_1\sqcup N_2)$, and hence any transverse vector field in $\tilde \Sigma_{AP}$ can be modified on $N_1\sqcup N_2$ to satisfy (1) and (2), and scaled on its complement to have pointwise norm less than 2.  Assertions (3) and (4) follow from the fact that in $\RR^4$, $\exp_p( r{\bf v})=p+r{\bf v}$, see Section \ref{pushoff2}.
 \end{proof}

\subsection{The building block $(R,\tilde h_R)$}

Choose  $0<\delta<\tfrac {\ep ^2} {8}$     small enough so that  the push-off
$$\tilde \Sigma_{AP,\delta}:=\exp(\delta\tfrac{\nu}{|\nu|}(\tilde \Sigma_{AP}))$$
is embedded.  
 The loop $\mu_{R,\delta}=\mu_{S,\delta}$ based at $e$ (\ref{mudel}) is the boundary of the normal disk to $\tilde \Sigma_{AP,\delta}$ at $\exp(\delta\tfrac{\nu}{|\nu|}(e))=(0,0,\tfrac{\delta}{\sqrt{2}},\tfrac{\delta}{\sqrt{2}})$. 
Lemmata  \ \ref{pushoff},    and    \ref{lem3.1}  and Proposition \ref{lem1.33} imply the following. 
 
\begin{thm}\label{R} There exists an embedding
$$\tilde h_{R}: \Sigma\times D^2\to R$$ satisfying
\begin{enumerate}
\item $\tilde h_{R}(\Sigma\times 0)=\tilde\Sigma_{AP,\delta}$, $\tilde h_{R}(\Sigma\times 1)=\tilde\Sigma_{AP}$,
$\tilde h_{R}(\sigma,1)=e$, 
\item 
$\tilde h_{R}(a_1,1)= \zeta,~\tilde h_{R}(b_1,1)=\tau ,~\tilde h_{R}(a_2,1)= y,~\tilde h_{R}(b_2,1)=\bar  \xi,
$
and $\tilde h_{R} (\sigma\times S^1)=\mu_{R,\delta},$ and
 \item the composite 
 $$ \pi_1(\Sigma\times\{1\},\sigma)\xrightarrow{\tilde h_{R } }
 \pi_1(\tilde\Sigma_{AP},e)
\xrightarrow{i_R} \pi_1(R,e)$$ is given by
$$ a_1\mapsto z, ~ b_1\mapsto t, ~ a_2\mapsto   y, ~ b_2\mapsto \bar x^2.$$\qed
\end{enumerate}
\end{thm}

\subsection{Some relations in $\pi_1(R\setminus \tilde\Sigma_{AP,\delta})$}

Unlike    the surface $\tilde \Sigma_S$ in $S$, the surface $\tilde \Sigma_{AP}$ in $R$ is not equipped with a smooth dual 2-sphere (the exceptional 2-sphere intersects it in two points), and hence application of the SVK theorem requires
establishing some relations in $\pi_1(R\setminus \tilde \Sigma_{AP,\delta})$.  

\medskip

 \noindent{\bf Proof strategy.} The following technical lemmata assert  that various loops and homotopies in $R$ miss $\tilde\Sigma_{AP,\delta}$. The proofs in each case are different, but generally speaking, the strategy is first to show that the loop or homotopy $L$ misses an $\ep$-neighborhood of $T_x\sqcup T_y\sqcup p\sqcup \mathfrak{E}$ (with $\mathfrak{E}$ the exceptional curve). Notice that the complement of $\Nbd{T_x\sqcup T_y\sqcup p\sqcup p'}$ in $\TT^4$ is equal to the complement of $\Nbd{T_x\sqcup T_y\sqcup p\sqcup \mathfrak{E}}$ in $R$.

The second step is to use coordinates in $\TT^4$ to check that $L$ misses 
$$\tilde\Sigma_{AP,\delta}\cap \TT^4\setminus \Nbd{T_x\sqcup T_y\sqcup p\sqcup p'}.$$
 \begin{wrapfigure}
[6]{r}{0.4\textwidth}
  \centering
\includegraphics[width=2.5in]{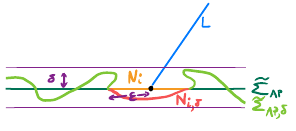}
 \vskip-.2in
\caption{\label{fig60fig}\color{red} }
 \end{wrapfigure} 
This second step then breaks down further: first,  Lemma \ref{framing} is used to show that $L$ misses $N_{1,\delta}\sqcup N_{2,\delta}\subset \tilde\Sigma_{AP,\delta}$, and second,  
the distance from $L$ to $\tilde\Sigma_{AP}\setminus (N_1\sqcup N_2\sqcup \Nbd{p\sqcup p'})$ is shown to be greater than $\delta$, and therefore $L$ misses the $\delta$ push-off of $\tilde\Sigma_{AP}\setminus (N_1\sqcup N_2\sqcup \Nbd{p\sqcup p'})$. 
Figure \ref{fig60fig} gives a ``side view'' schematic of the second step.  It should be compared to Figures \ref{fig47fig} and \ref{fig59fig}.\footnote{Additional figures which may help the reader follow the proofs can be found in Appendix \ref{AF}.}

 \begin{lem}\label{9.2}\hfill
 \begin{enumerate}
\item The loops $\xi, y,\zeta, \tau$ and $\mu_z$ lie on $\tilde \Sigma_{AP}$ and hence in
$R\setminus \tilde \Sigma_{AP,\delta}$.

\item The loop $\mu_{R,\delta}$   and the loops $x, z, t,   \mu_x, \mu_y$ lie  in $R\setminus \tilde \Sigma_{AP,\delta}$.

\end{enumerate}
\end{lem}
\noindent{\em Proof.} 
Since $\tilde\Sigma_{AP}\subset R\setminus \tilde\Sigma_{AP,\delta}$, (1) follows.  For (2), the loop $\mu_{R,\delta}$ lies on the boundary of a tubular neighborhood of  $\tilde \Sigma_{AP,\delta}$, hence misses $\tilde \Sigma_{AP,\delta}$.  

That $x,z,$ and $t$ miss $\tilde \Sigma_{AP,\delta}$ can be easily seen using Figure \ref{fig44fig} and Lemma \ref{framing}.    More formally, recall that $T_z=\{(-\tfrac\pi 2,\tfrac\pi 2,z,t)\mid z,t\in S^1\}$.  The loop $x(\theta)=(\theta,0,0,0)$ satisfies
${\rm dist}(x,T_z)\ge\tfrac\pi 2$. Therefore, $x$ misses $T_{z,0,\delta}$ and $\Nbd{p\cup p'}$, and hence also the annulus $A_\delta$.  The 
 distance from $x$ to any point in $E_{AP}\setminus N_1$ is greater than $\ep>\delta$, since the projections of $x$ and $\xi$ to $\TT^2\times\{(0,0)\}$ have the same image  (see Figure \ref{fig44fig}) and therefore, $x$ misses $\big(E_{AP}\setminus(\Nbd{p\cup p'}\cup N_1)\big)_\delta.$ Finally, $x$ misses $N_{1,\delta}$ by Lemma \ref{framing}.  Thus $x$  misses $\tilde \Sigma_{AP,\delta}$.

 Next, $z(\theta)=(0,0,\theta,0)$ satisfies    ${\rm dist}(z,T_z)\ge\tfrac\pi 2$ by comparing 2nd coordinates.  
It meets $E_{AP}$ in precisely the point $e$, since the only other point in $E_{AP}$ with first two coordinates zero lies in $\{g>0\}$ (see Figure \ref{fig44fig}) and hence has fourth coordinate positive. Since $e\in N_1$, $z$ misses $N_{1,\delta}$. Moreover, since $N_1$ contains an $\ep$ neighborhood of $e$, the distance from $z$ to $E_{AP}\setminus (N_1\cup \{g>0\})$ is greater than $\ep$.   Hence $z$ misses $\tilde \Sigma_{AP,\delta}$.

Next, the loop $t(\theta)=(0,0,0,\theta)$ satisfies  ${\rm dist}(t,T_z)\ge\tfrac\pi 2$ by comparing 2nd coordinates.  
It meets $E_{AP}$ in precisely the point $e$ and hence has distance greater than $\delta $ from $\tilde \Sigma_{AP}\setminus N_1$. It  meets $N_{1}$ precisely at $e$, which lies in $\tilde \Sigma_{AP}\subset R\setminus \tilde \Sigma_{AP,\delta}$. Therefore $t$ misses $\tilde \Sigma_{AP,\delta}$.

That the loop $\mu_x$ misses $\tilde \Sigma_{AP,\delta}$ follows simply because the line segment  from $e$ to $\partial \Nbd{T_x}$ misses $N_{1,\delta}$.   See Figure \ref{fig44fig}.
 The same argument works for $\mu_y$, once one notices that the line segment from  $e$ to $\partial \Nbd{T_y}$ misses $N_{1,\delta}$ and $\{g>0\}$.
\qed

 \begin{lem} \label{Ball} Any loop in $R\setminus \tilde \Sigma_{AP,\delta}$ based at $e$ which lies within $5\ep$ of the line segment $\rho_1$
 is homotopic in $R\setminus \tilde \Sigma_{AP,\delta}$ to $\mu_{R,\delta}^k$ for some integer $k$.
 \end{lem}

 \begin{wrapfigure}
[9]{r}{0.4\textwidth}
  \centering
\includegraphics[width=2.8in]{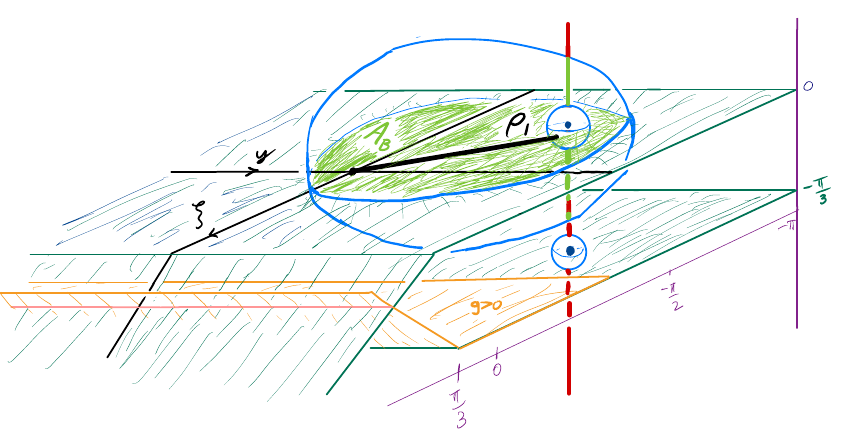}
 \vskip-.2in
\caption{\label{fig51fig}\color{red} }
 \end{wrapfigure} 
 \noindent{\em Proof.}  The (convex) $5\ep$ neighborhood $B=\Nbd{\rho_1}$ of the line segment $\rho_1$ in $\TT^4_\sur$ is diffeomorphic to a 4-ball, with $A_{B }:=\tilde \Sigma_{AP}\cap B$ the resolution of the intersecting disks
$E_{AP}\cap B$ and $T_z\cap B$. 
Figure \ref{fig51fig} shows the intersection of $B$ with $\TT^3\times\{0\}$, as well as that part of the annulus $A_B$ which lies outside $\Nbd{p}$.

Let $A_{B,\delta}= \tilde\Sigma_{AP,\delta}\cap B$.   For all small enough $\delta$,  $A_{B,\delta}$ is a properly embedded annulus in  $B$, isotopic into   $\partial B$,  with boundary isotopic to the Hopf link.  

Suppose $F^+$ is obtained by pushing  the interior of a connected compact surface $F\subset S^3$
into the interior of $B^4$. The fundamental group 
$\pi_1(B^4\setminus F^+)$ is infinite cyclic, generated by any meridian. This follows from
\cite[Proposition 6.2.1]{GS},  since $F^+$ admits a Morse function with one minimum value and maximum set $\partial F$. Hence
$\pi_1(B\setminus A_{B,\delta},e)$ is infinite cyclic, generated by  $\mu_{R,\delta}$. \qed

  \begin{lem} \label{dicyAF2}

There exist   integers $k_1,k_2$ so that the loops
 $ x,  y, z,  t ,  \zeta,\tau,\mu_x,\mu_y,\mu_{R,\delta}$  
in  $R\setminus \tilde\Sigma_{AP,\delta}$
satisfy:

\begin{tasks}[style=enumerate, item-format={\normalfont}, after-item-skip=-1mm](4)
 \task[(a)]~ $\mu_{R,\delta}=[y,\bar x]$,
 \task[(b)]~ $[y,\mu_{R,\delta}]=1$,
 
\task[(c)]~
  $\tau=\mu_{R,\delta}^{k_1}   t \mu_{R,\delta}^{-k_1},$ 
\task[(d)]~ $[\tau, y]=1$, 
\task[(e)]~$[z,\bar y]=\mu_x$, 

\task[(f)]~ $ 
\zeta=\mu_{R,\delta}^{k_2}   z \mu_{R,\delta}^{-k_2}$,

\task[(g)]~   $x=\mu_x$
\end{tasks}
 in  $\pi_1(R\setminus \tilde\Sigma_{AP,\delta},e).$
\end{lem}

 \noindent{\it Proof.} 
 \begin{wrapfigure}
[9]{r}{0.35\textwidth}
  \centering
   \vskip-1in
\includegraphics[width=2.3in]{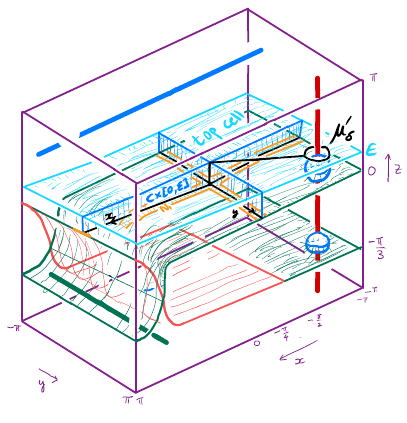}
 \vskip-.2in
\caption{\label{fig47fig}\color{red}}
 \end{wrapfigure} 
(a). Construct a map   $\alpha\colon D^2\to \TT^3\times \{0\}$ as follows.  The torus $\TT^2\times \{(\ep,0)\}$ lies in $R$ and intersects $\tilde \Sigma_{AP,\delta}$ in exactly one point,  transversely. Its top cell has boundary mapping to (the non-embedding) $c\times\{(\ep,0)\}$, with $c$ the commutator map $[y,\bar x]$. Attach the collar $ c\times [0,\ep]$ to the top cell to construct the  map $\alpha$.  

 The collar  $c\times [0,\ep]$ lies in $\{t=0\}$ and hence misses 
$N_{1,\delta}$.
It intersects $N_1$ in a subset of $\xi\cup y$ and hence   the distance between $c\times [0,\ep]$ and $\tilde \Sigma_{AP}\setminus  N_1$ is at least $\ep\ge 8\delta$. Hence $c\times [0,\ep]$ misses $\tilde \Sigma_{AP,\delta}\setminus N_{1,\delta}$. Thus $c\times [0,\ep]$ misses $\tilde \Sigma_{AP,\delta}$.
 See Figure \ref{fig47fig}, where $c\times[0,\ep]$ and    the top cell are drawn.

Consider the based loop $\mu_\delta'$ indicated in Figure \ref{fig47fig}, obtained from $\mu_z$ (Figure \ref{fig44fig}) by shifting its $z$ coordinate up by $\ep$; with the  endpoint $(0,0,\ep,0)$ connected to $e$ by the line segment joining them.  The loop $\mu_\delta'$ is the image under $\alpha$ of a loop in $D^2$, proving that $\mu_\delta'=[y,\bar x]$ in $\pi_1(R\setminus \tilde \Sigma_{AP,\delta})$.  Note that $\mu_\delta'$ lies in the $5\ep$ neighborhood of $\rho_1$, and is an oriented meridian of $\tilde \Sigma_{AP,\delta}$. Lemma \ref{Ball} implies that $\mu_\delta'=\mu_{R,\delta}$.

 \medskip

(b).  Since $y=\tilde h_R(a_2,1)$ and  $\mu_{R,\delta}=\tilde h_R(\sigma\times S^1)$, it follows that $[y,\mu_{R,\delta}]=1$.

\medskip

(c). Recall that $\tilde\tau(\theta)=(-\tfrac\pi 2, \tfrac\pi 2,-3\ep,-3\ep +\theta),~\theta\in S^1.$  Assume for a moment that the annuli
$$A_{\tau,1}=\{ (-u\tfrac\pi 2, u\tfrac\pi 2,-3\ep,-3\ep u +\theta)\mid u\in [0,1],\theta\in S^1\}\text{ and } A_{\tau,2}=\{ (0, 0 ,-3\ep u,\theta)\mid u\in [0,1],\theta\in S^1\}$$
miss $\tilde \Sigma_{AP,\delta}$. These provide a free homotopy between $\tilde \tau$ and $t$ in $R\setminus \tilde \Sigma_{AP,\delta}$.  The paths $$(-u\tfrac\pi 2, u\tfrac\pi 2,-3\ep,-3\ep u )\text{ and } (0, 0 ,-3\ep u,0),~u\in [0,1]$$
lie in  the $5\ep$ neighborhood of $\rho_1$ and lie on $A_{\tau,1}$ and $A_{\tau,2}$.   These compose with $\rho_1\rho_2\rho_3$ to form a loop based at $e$, which, by Lemma \ref{Ball}, equals $\mu_{R,\delta}^{k_1}$ in $\pi_1(R\setminus \tilde \Sigma_{AP,\delta},e)$ for some integer $k_1$.  It follows that  $\tau=\mu_{R,\delta}^{k_1} t\mu_{R,\delta}^{-k_1}$.

We return to show these annuli miss $\tilde \Sigma_{AP,\delta}$.  The annulus $A_{\tau,1}$ misses  $T_{z,0,\delta}$ since $N_2$ is an $\ep$ neighborhood of $\tilde \tau$, and $N_{2,\delta}=N_2+(-\delta,0,0,0)$.  It misses the $\ep$ neighborhoods of $p$ and $p'$ by comparing third coordinates. 
 It misses an $\ep$ neighborhood of $E_{AP}$ since the path $[0,1]\ni u\mapsto (-u\tfrac\pi 2,u\tfrac\pi 2,-3\ep)$ lies between the two sheets of $E_{AP}\cap \{x\in [-\pi,0]\}$, at distance at least $\ep$.  
 Hence $A_{\tau,1}$ misses $\tilde \Sigma_{AP,\delta}$.  The annulus $A_{\tau,2}$ clearly misses an $\ep$ neighborhood of $T_z\cup p\cup p'$, and since $N_1$ contains an $\ep$ neighborhood of $e$, it also misses an $\ep$ neighborhood of $\tilde\Sigma_{AP}\setminus N_1$ and hence misses $\tilde\Sigma_{AP,\delta}\setminus N_{1,\delta}$. Finally, $A_{\tau,2}$  misses $N_{1,\delta}$ since  
Lemma \ref{framing} shows that the distance between $N_{1,\delta}\cap \{-\pi<x<\ep\}$ and  $(0,0,-3\ep u,\theta)\in A_{\tau,2}$ is at least $\tfrac{\delta}{\sqrt{2}}$ (note that
$\tfrac{\pi}{3}-3\ep>\ep>\delta$, so that $A_{\tau,2}$ misses the lower sheet of $N_{1,\delta}$).

\medskip 
 
(d).  The distance from the torus 
$\{(0, y, 0, t)\mid y,t\in S^1\}$   to  $T_x , T_y$  and $T_z$ is greater than $\ep$, by comparing third  and first coordinates. Its distance to $p,p'$  is greater than $\ep$ by comparing first coordinates.  It intersects $E_{AP}$ in    the curve $y$, whose distance to $E_{AP}\setminus N_1$ is greater than $\ep$. Since every point in $N_{1,\delta}\cap \{|x|<\ep, |z|<\ep\}$ has non-zero third coordinate by Lemma \ref{framing}, this torus misses $N_{1,\delta}$ and hence misses $\tilde\Sigma_{AP,\delta}$. It follows that $[y,t]=1$  in  $\pi_1(R\setminus \tilde\Sigma_{AP,\delta},e).$
By (b) and (c), 
 $[\tau, y]=\mu_{R,\delta}^{k_1}[t, y]\mu_{R,\delta}^{-k_1}=1.$
 
\medskip

(e).
The distance from the torus 
$\{(0, y, z,0)\mid y,z\in S^1\}$   to  $T_y$  and $T_z$ is greater than $\ep$, by comparing first coordinates. It misses $N_{1,\delta}$ and has distance at least $\tfrac{\ep}{4}$ from $\{g>0\}\subset E_{AP}$, and hence misses $\tilde\Sigma_{AP,\delta}$. It intersects $T_x$ transversely in the point $(0,-\tfrac\pi 2, \tfrac\pi 2,0)$, proving that $[z,\bar y]=\mu_x$ in $\pi_1(R\setminus \tilde\Sigma_{AP,\delta},e).$

\medskip

 (f). Recall that $\tilde \zeta(\theta)=(-\tfrac\pi 2, \tfrac\pi 2, -3\ep +\theta, -3\ep), ~\theta\in S^1$. Assume for the moment that the annuli
$$
A_{\zeta,1}=\{(-u\tfrac\pi 2, u\tfrac\pi 2, -3\ep u +\theta, -3\ep), ~\theta\in S^1, u\in [0,1]\}\text{ and }
A_{\zeta,2}=\{(0,0,\theta, -3\ep u), ~\theta\in S^1, u\in [0,1]\}
$$
miss $\tilde \Sigma_{AP,\delta}$.

These provide a free homotopy between $\tilde \zeta$ and $z$ in $R\setminus \tilde \Sigma_{AP,\delta}$.  The paths $$(-u\tfrac\pi 2, u\tfrac\pi 2,-3\ep u,-3\ep  )\text{ and } (0, 0 ,0, -3\ep u),~u\in [0,1]$$
lie in  the $5\ep$ neighborhood of $\rho_1$ and lie on $A_{\zeta,1}$ and $A_{\zeta,2}$.   These compose with $\rho_1\rho_2\rho_3$ to form a loop based at $e$, which, by Lemma \ref{Ball}, equals $\mu_{R,\delta}^{k_2}$ in $\pi_1(R\setminus \tilde \Sigma_{AP,\delta},e)$ for some integer $k_2$.  It follows that  $\zeta=\mu_{R,\delta}^{k_2} z\mu_{R,\delta}^{-k_2}$  in $\pi_1(R\setminus \tilde\Sigma_{AP,\delta},e).$

It remains to show  that $A_{\zeta,1}$ and $A_{\zeta,2}$ miss $\tilde\Sigma_{AP,\delta}$.  For $A_{\zeta,1}$, note that $A_{\zeta,1}$ intersects $\tilde\Sigma_{AP}$ precisely in the loop $\tilde\zeta=\{(-\tfrac\pi 2, \tfrac\pi 2, -3\ep  +\theta, -3\ep)\mid ~\theta\in S^1\}
$, which lies in $N_2\subset T_{z,0}$. The distance  from any point in $A_{\zeta,1}$ to $\tilde\Sigma_{AP}\setminus  N_2 $ is greater than $\ep$, and since $ A_{\zeta,1}$ misses $N_{2,\delta}$ by Lemma \ref{framing}, $ A_{\zeta,1}$ misses 
$\tilde\Sigma_{AP,\delta}$. For  
$A_{\zeta,2}$, note that $A_{\zeta,2}$ intersects $\tilde\Sigma_{AP}$ precisely at the point $e$. Hence  the distance from $A_{\zeta,2}$ to $\tilde \Sigma_{AP}\setminus N_1$ is at least $\ep$. Since $A_{\zeta,2}$ misses $N_{1,\delta}$ by Lemma \ref{framing}, $ A_{\zeta,2}$ misses 
$\tilde\Sigma_{AP,\delta}$.

  \begin{wrapfigure}
[15]{r}{0.35\textwidth}
  \centering
  \vskip-.2in
\includegraphics[width=2.4in]{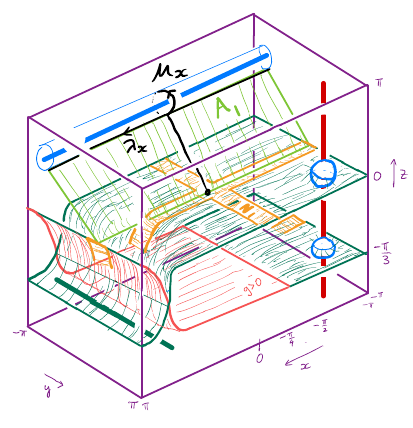}
\vskip-.2in
\caption{\label{fig59fig}}
 \end{wrapfigure} 
 (g).  Define the loop $$\lambda_x:=\{(x, -\tfrac\pi 2+\tfrac{\ep}{\sqrt{2}}, \tfrac\pi 2-\tfrac{\ep}{\sqrt{2}},0)\mid x\in S^1\},\hskip2in$$
and connect it to the base point by the line segment $\rho_x$ from $e$ to $(0,-\tfrac\pi 2+\tfrac{\ep}{\sqrt{2}}, \tfrac\pi 2-\tfrac{\ep}{\sqrt{2}},0)$. This is the same line segment used to base the meridian $\mu_x$, and therefore,  
the surgery relation  for $T_x$  in $R\setminus \tilde\Sigma_{AP,\delta}$ is $\mu_x=\lambda_x$.

The annulus 
 $$A_1:=\{(x, -u, u,0)\mid x\in S^1, 0\leq u\leq \tfrac\pi 2-\tfrac{\ep}{\sqrt{2}}\}\subset \TT^3\times\{0\}\hskip2.5in$$  
 contains $\lambda_x$,
 has distance at least $\ep$ from $T_x, T_y$, and $T_z$, 
has distance greater than $\tfrac{\ep}{\sqrt{2}}>\delta$ from $E_{AP}\setminus N_1$, and
 misses $N_{1,\delta}$ since  the last coordinate of $A_1$ is zero.  
Hence $A_1$  misses $\tilde\Sigma_{AP,\delta}$,  and  provides a  based homotopy in $R\setminus \tilde\Sigma_{AP,\delta}$ from $x$ to $\lambda_x$, from which it follows by the surgery relation that $x=\lambda_x=\mu_x$.  See Figure \ref{fig59fig}.
\qed

 \subsection{The surface sum $Z_{AP}=R\#_\Sigma P$}
 
  We use the same gluing $\Psi$ (Equation (\ref{glue2})) as used in Definition \ref{defZ} to glue the building blocks $(R,\tilde h_R)$ and $(P,\tilde h_P)$  together.
\begin{df}\label{defZap} Define the {\em surface sum of $R$ and $P$ along $\tilde\Sigma_{AP}$ and $\Sigma_P$} to be
  
 \begin{equation}\label{ZAP}Z_{AP}=\big( R\setminus\tilde h_{R}(\Sigma\times D^2)\big)\cup_{\psi}\big(P\setminus\tilde h_{P}(\Sigma\times D^2)\big)
\end{equation}where $$\psi:\partial \Nbd{\tilde \Sigma_{AP,\delta}}\xrightarrow{\tilde h_R^{-1}}\Sigma\times S^1\xrightarrow{\Psi\times {\rm -Id}}
\Sigma\times S^1
\xrightarrow{\tilde h_P}\partial \Nbd{\Sigma_{P,\delta}}.$$

\end{df}

 Theorems \ref{R},  \ref{thmP}, and SVK imply that the relations 
 $${\rm Sum}_{AP}:= \{\zeta=y'\bar x'\bar y', \tau=\bar y', y=y''\bar x''\bar y'', \bar\xi=\bar y'',~\mu_{R,\delta}=\mu_{P,\delta}^{-1}\}$$ hold  in $\pi_1(Z_{AP})$.

\begin{thm}\label{Zap}
The closed 4-manifold $Z_{AP}$ is  simply connected. \end{thm}

\begin{proof}
Set $W=R\setminus\tilde h_{R}(\Sigma\times D^2)$.
The loops  $y=\tilde h_R(a_2,1)$  and $\tau=\tilde h_R(b_1,1)$ lie  in the boundary of $W$, and   
Lemma \ref{dicyAF2} (d) asserts that $[ \tau,y]=1$ in $\pi_1(W,e)$. 
Theorem \ref{trick} applies to show  that    
$[\zeta, y]=[\tilde h_R(a_1,1),\tilde h_R(a_2,1)]=1= [y'\bar x'\bar y', y'' \bar x''\bar y'']$ in $\pi_1(Z_{AP},e)$.  

Using Lemma \ref{dicyAF2} (b), (e), and (f), 
$1=\mu_{R,\delta}^{-k_2}[\zeta,\bar y]\mu_{R,\delta}^{k_2}=
[z,\bar y]=\mu_x$. 
    Lemma \ref{dicyAF2}   (g) implies that
$ x=1 $  in $\pi_1(Z_{AP},e)$.   Lemma \ref{dicyAF2}   (a) now shows that 
$\mu_{R,\delta}=1$ in $\pi_1(Z_{AP},e)$, and Sum$_{AP}$ gives $\mu_{P,\delta}=1$.

By the SVK theorem, and the handle attaching argument   (see Appendix \ref{CP})
$$ 
\pi_1(R\setminus \tilde \Sigma_{AP,\delta},e)/\langle\langle \mu_{R,\delta}\rangle\rangle=\pi_1(R,e)\text{ and }\pi_1(P\setminus  \Sigma_{P,\delta},e)/\langle\langle \mu_{P,\delta}\rangle\rangle=\pi_1(P,e).
$$   
This implies that  the SVK diagram factors:
\[
\begin{tikzcd}
&\pi_1(\Sigma\times S^1)\arrow[dl]\arrow[dr]\\
\pi_1(R\setminus \tilde\Sigma_{AP,\delta})\arrow[d,"i_R" ]&&\pi_1(P\setminus \Sigma_{P,\delta})\arrow[d,"i_P"]\\
\pi_1(R)\arrow[dr]&&\pi_1(P)\arrow[dl]\\
&\pi_1(Z_{AP})&
\end{tikzcd}\]
with   $i_R$ and $i_P$ surjections whose kernels  are normally generated by $\mu_{R,\delta}$ and $\mu_{P,\delta}$.   
It follows that  $\pi_1(Z_{AP},e)$ is generated by the twelve loops
$$
  x,  y,  z,   t,x',y',z',t',x'',y'',z'',t''$$
and these satisfy, in addition to $x=1$,  the relations Rel, Rel$'$, Rel$''$, Gl, and Sum$_{AP}$.  Moreover, this implies (Proposition \ref{thm1})
that  the loop $\mu_y$ satisfies
$ \mu_y=[x,\bar z ]
$
and therefore, $1=\mu_y=t$ by Sur.

\medskip

We finish the proof by showing that the remaining ten generators are also trivial in $\pi_1(Z_{AP})$.
 Since $t=1$,  Lemma \ref{dicyAF2} (c)  shows $\tau=1$, which by Sum$_{AP}$ implies $y'=1$. Then Sur$'$ shows $x'=1$, and then $t'=1$.
Next, Gl implies that $z''=1$, which, using Sur$''$, shows $x''=1=t''$. Since $t''=1$, Gl implies that $z'=1$.
Sum$_{AP}$   shows that  $y=y'' \bar x'' \bar y''=1$. 
 Lemma \ref{dicyAF2} (f), $x'=1$,   and Sum$_{AP}$   show that
$ 1=y'\bar x'\bar y'=\zeta=\mu_{R,\delta}^{k_2} z\mu_{R,\delta}^{-k_2}$, so that $\zeta=1=z$.

Finally, $\xi=x^2$ in $\pi_1(R)$, so that, using Sum$_{AP}$,
$y''= \xi=1$ in $\pi_1(Z_{AP})$.  All  generators are killed by the relations, and hence $Z_{AP}$ is simply connected.
\end{proof}

The Euler characteristic of $Z_{AP}$ equals 5 and its signature equals $-1$. Its second Betti number is three and hence the intersection form is odd. Since $Z_{AP}$ is smooth, it has vanishing Kirby-Siebenmann invariant.  Hence by Freedman's theorem \cite{Freedman} $Z_{AP}$ is homeomorphic to $\CP ^2\# 2\overline{\CP }^2$.

\appendix

\section{Exotica} \label{exotic}
We briefly explain, using the language of symplectic topology, the relationship between Theorems \ref{thm5.1}  and \ref{Zap} and the existence of 4-manifolds homeomorphic but not diffeomorphic to $\CP ^2\# 3\overline{\CP }^2$ and $\CP ^2\# 2\overline{\CP }^2$. 

The 4-manifold $\TT^4_\sur$ admits a symplectic structure since it is obtained by Luttinger surgeries on $\TT^4$ \cite{luttinger}. Blowing up can be done symplectically, so that
$S$ and $R$ are symplectic. The surfaces $ \Sigma_{P},\tilde \Sigma_{S},$ and $\tilde \Sigma_{AP}$ are symplectic, as are the nearby surfaces including $ \Sigma_{P,\delta},\tilde \Sigma_{S,\delta},$ and $\tilde \Sigma_{AP,\delta}$ for small enough $\delta$.  Hence  the fiber sums $Z$, $Z_{AP}$  are symplectic. The deformation class of the resulting symplectic structure on $Z$ and $Z_{AP}$ is independent of all choices made in the construction \cite{Gompf}.
A symplectic 4-manifold is  symplectically minimal if and only if it contains no embedded symplectic  2-spheres of square $-1$.  
\begin{lem} The closed symplectic 4-manifolds $Z$ and $Z_{AP}$ are symplectically minimal.
\end{lem}
\noindent{\em Proof.}
First, Lemma 5 of \cite{BK} states that $\TT^4_\sur$ is aspherical (this follows from the fact that $\TT^4_\sur$ is a circle bundle over a torus bundle over the circle).

The manifold $P$ is symplectically minimal. To see this, note that since $\TT^4_\sur$ is aspherical, it contains no $-1$ symplectic 2-spheres. Asphericity also  implies that $\TT^4_\sur$ is not homotopy equivalent to a 2-sphere bundle over a surface. Hence Usher's theorem \cite{Usher} implies that  the symplectic sum $P=\TT^{4\prime}_\sur \#_{T_z'=T_z''} \TT^{4\prime\prime}_\sur $,  is symplectically  minimal.

  The Hopf exact sequence  
 $$\pi_2(\TT^4_\sur\#\bCP^2)\to H_2(\TT^4_\sur\#\bCP^2)\to H_2(\pi_1(\TT^4_\sur))\to 0$$
shows that the spherical classes in $H_2(\TT^4_\sur\#\bCP^2)$
are precisely the multiples of the exceptional class $[\mathfrak{E}]$. Since $\tilde\Sigma_{AP}\cdot n[\mathfrak{E}]=2n$, the complement $R\setminus \tilde\Sigma_{AP}$ contains no homologically essential   2-spheres, and in particular no symplectic $-1$ spheres.  

Similarly, the spherical classes in $H_2(S)$ are precisely the linear combinations of the two exceptional 2-spheres.  
Note that $(m\mathfrak{E}_1+n\mathfrak{E}_2)^2=-(m^2+n^2)$ and $\tilde \Sigma_S\cdot(m\mathfrak{E}_1+n\mathfrak{E}_2)=m+n.$  Hence if a spherical class has square $-1$
it intersects $\tilde  \Sigma_S$.  In particular, 
the complement $S\setminus \tilde \Sigma_{S}$ contains no $-1$ symplectic 2-spheres.

The manifolds $P, S, $ and $ R$ do not admit   the structure of
an $S^2$ bundle-with-section over a genus two surface, since the total space of such a bundle has Euler characteristic $-4$, whereas $\chi(P)=0$, $\chi(R)=1$, and $\chi(S)=2$.

Applying Usher's theorem again shows that the symplectic sums $Z=P\#_{\Sigma_P=\tilde\Sigma_S} S$
and  $Z_{AP}=P\#_{\Sigma_P=\tilde\Sigma_{AP}} R$
are symplectically minimal.
 \qed

\medskip

Results of Taubes   \cite{Taubes, Taubes2}, extended to the $b^+=1$ case by Li and Liu \cite{Li, Liu, Li-Liu, Li-Liu2}, assert   that symplectic minimality implies {\em smooth minimality} in dimension 4. One concludes the following.
 
\begin{thm} The manifold  $Z$  is a closed, simply-connected, 
minimal symplectic 4-manifold  homeomorphic, but not diffeomorphic  to $\CP ^2\# 3\overline{\CP }^2$.  The manifold $Z_{AP}$   is a closed, simply-connected, 
minimal symplectic 4-manifold  homeomorphic, but not diffeomorphic  to    $\CP ^2\# 2\overline{\CP }^2$.
\qed
\end{thm}

\subsubsection{Reverse engineering}
To give the constructions in this article some wider context, we provide the following imperfect discussion of 
Fintushel-Stern's   {\em reverse engineering} method for symplectic 4-manifolds \cite{FS1, FS2, FPS}. For the  geography  problem, one starts with a  minimal symplectic 4-manifold $X$ with 
some prescribed Euler characteristic and signature, built, perhaps, by taking surface sums of simple  symplectic 4-manifolds, such as products of surfaces.
One then searches  for  Lagrangian tori  in $X$ to Luttinger surger,  in hopes of transforming $X$  into a    symplectic 4-manifold $Y$ with $\pi_1(Y)=0$. If such surgeries can be found, and the result $Y$ is symplectically irreducible,    then 
as in the examples of $Z$ and $Z_{AP}$,  one produces an exotic copy of some standard 4-manifold.   Indeed, this philosophy is what leads to the building block approach taken here, as well as in \cite{AP1,AP2,BK} and much similar work in the literature.

Reverse engineering does  more, however, addressing the botany problem. Using the Morgan-Mrowka-Szabo
theorem \cite{MMS}, Fintushel-Stern show how to compute some of  the Seiberg-Witten invariants of  an infinite family of  4-manifolds obtained by torus surgery.    It is beyond the scope of this article to describe this aspect of their work. But,
using it, one can show the following.  Construct an infinite family of   smooth manifolds  $P_n$, $n=1,2,3,\cdots$
by replacing the  surgery in the construction of $P$  with corresponding  relation $t''=[x'',\bar z'']$, by the {\em non-Luttinger} torus surgery giving the relation $t''=[x'',\bar z'']^n$.
The reader can easily check, by rereading the last paragraph of the proofs of Theorems \ref{thm5.1}  and \ref{Zap} but inserting this altered relation, that $P_n\#_\Sigma S$ and $P_n\#_\Sigma R$ are simply connected for all $n\in \NN$.
Reverse engineering can then be used to  show that  $\{P_n\#_\Sigma S\}_{n\in \NN}$   and $\{P_n\#_\Sigma R\}_{n\in \NN}$   each  contains   an infinite  subfamily of pairwise non-diffeomorphic   manifolds.  An argument which applies to both 
$Z$ and $Z_{AP}$ can be found on page 343 of \cite{BK3}.   Since $P_1=P$,  the first family contains $Z$ and the second $Z_{AP}$.

 \section{Cut-and-paste operations in 4-dimensions} \label{CP} The following operations are discussed in detail in any 4-manifold textbook.
\subsubsection{Torus surgery} 
Given a framed   embedding  $\alpha: T\times D^2\subset X$ of a 2-torus $T$ into  a 4-manifold $X$ and an isotopy class $\beta$ of an orientation-preserving  diffeomorphism  
$T \times S^1\to T\times S^1$ (i.e. $\beta\in SL(3,\ZZ)$), let  $X(\beta)$ denote the 4-manifold obtained from $X$ by  removing $\alpha( T\times D^2)$ and replacing it using $\alpha\circ\beta$:
$$
X( \beta)=X\setminus \alpha(T\times D^2)\cup_{\alpha\circ\beta}T\times D^2.
$$

\medskip
Turning the usual handle decomposition for $T\times D^2$  upside down reveals that 
$X(\beta)$ is obtained from $X \setminus  \alpha(T\times D^2)$ by attaching, in succession, one 2-handle, two 3-handles, and one 4-handle,  to $\partial\big(X\setminus  \alpha(T\times D^2))$. 
Hence, using  the SVK theorem, 
a presentation for $\pi_1(X(\beta),e)$ is obtained from one for $\pi_1\big(X\setminus \alpha(T\times D^2),e\big)$ by adding a word representing $m_\beta\in \pi_1\big(X\setminus \alpha(T\times D^2),e\big)$ to the list of relations, where $m_\beta$ denotes the boundary of the core of the 2-handle,  $\alpha\circ\beta(\{r\}\times \partial D^2)$,  connected to the base point by some path. The Euler characteristics and signatures  satisfy $\chi(X(\beta))=\chi(X)$ and $\sigma(X(\beta))=\sigma(X)$.

\subsubsection{Surface sum} Given closed oriented 4-manifolds $X,X'$,
a closed oriented surface $F$, and  an orientation-reserving diffeomorphism $\Psi:F\to F$,  and oriented framed
embeddings $\tilde h:F\times D^2\to X,~\tilde h':F\times D^2\to X'$ of $F$, the {\em surface sum of $X$ and $X'$ along $F$}   is the closed, oriented 4-manifold:
$$
X\#_{F,\Psi} X':=\big(X\setminus \tilde h(F\times D^2) \big) 
\cup _{\psi} \big(X'\setminus \tilde h'(F\times D^2) \big),
$$
where $\psi\colon \tilde h(F\times S^1) \to  \tilde h'(F\times S^1) $ is given by
$$\psi(\tilde h(f,e^{\theta\bbi}))=\tilde h'(\Psi(f),e^{-\theta\bbi}).$$

Then $\chi (X\#_FX')=\chi (X)+\chi (X')-2\chi(F)$ and $\sigma(X\#_FX')=\sigma(X)+\sigma(X')$.

\subsubsection{Blowing up at a point on a surface}
Given a smooth point $b$ on a smooth oriented (but not necessarily closed) surface $\Sigma$ in a smooth 4-manifold $X$, the {\em blow up} of the triple
$(X,\Sigma, b)$ is a smooth map   $c:(\tilde X,\tilde \Sigma, \mathfrak{E})\to(X,\Sigma, b)$,
where $(\tilde X,\tilde \Sigma)$ is the pairwise oriented connected sum
of $(X,\Sigma)$ and $(\overline{\CP }^2, \CP ^1)$ at $b$,
and $\mathfrak{E}$, the {\em exceptional curve}, is one of the   projective lines (a 2-sphere) in $\overline{\CP }^2$ which misses the ball pair $(B^4_\ep(b),D^2_\ep(b))$ at which the connected sum is taken. 
The smooth surface $\tilde \Sigma$, called the {\em proper transform} of $\Sigma$, meets $\mathfrak{E}$ transversely at one point. If $\Sigma$ is closed, then
 $\tilde \Sigma\cdot \tilde \Sigma=
 \Sigma\cdot \Sigma-1.$
 
 The normal $D^2$ bundle of $\mathfrak{E}$ in $\tilde X$  has Euler number $-1$ and boundary the 3-sphere. 
The smooth
{\em blow down} map $c:\tilde X\to X$  collapses $\mathfrak{E}$ to the point $b$, and  $c:\tilde X \setminus \mathfrak{E}\to X\setminus \{b\}$ is a diffeomorphism. 
In addition,  $c(\tilde \Sigma)=\Sigma$, $c:\tilde \Sigma\to \Sigma$ is a diffeomorphism, and $c^{-1}(\Sigma)=\tilde\Sigma\cup \mathfrak{E}$.

 If $\Sigma'\subset X$ is another smooth surface which
 intersects $\Sigma$ transversely once  at the point $b$, then the proper transforms $\tilde \Sigma$ and $\tilde \Sigma'$ in $\tilde X$ are disjoint.

 \medskip
 
   By the SVK theorem, $c:\pi_1(\tilde X)\to \pi_1(X)$ is an isomorphism. The presence of the   2-sphere $\mathfrak{E}$ dual to $\tilde \Sigma$ shows that the inclusion
  $\pi_1(\tilde X\setminus \tilde \Sigma)\to \pi_1(\tilde X)$ is an isomorphism, unlike, a priori, the inclusion $\pi_1( X\setminus  \Sigma)\to \pi_1( X)$.
\subsubsection{Resolving a transverse double point}
 Any positive transverse intersection point of oriented surfaces in an oriented 4-manifold can be locally parameterized by
$\{zw=0\}$ in $\CC^2$, a union of two transverse complex lines. These meet the unit 3-sphere in an oriented Hopf link. This Hopf link bounds an unknotted annulus $A^\flat$ in $S^3$ which can be pushed radially into $B^4$ to obtain a properly embedded annulus $A$. Removing the two intersecting 2-disks $B^4\cap \{z=0\}$ and $B^4\cap \{w=0\}$ and replacing them by $A$ results in a smooth surface with Euler characteristic reduced by
$2$.   The fundamental group $\pi_1(B^4\setminus A)$ is infinite cyclic, generated by any meridian \cite[Proposition 6.2.1]{GS}.

\section{Proof of Proposition \ref{thm1}}\label{prop21}

  \begin{wrapfigure}[9]{r}{0.25\textwidth}
\begin{center}
\vskip-.4in
\includegraphics[width=2in]{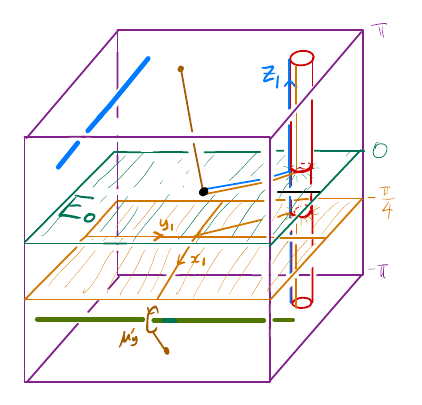}
 \vskip-.2in
\caption{\label{fig48fig}  }
\end{center}
\end{wrapfigure}
\noindent{\em Proof.}  Figure \ref{fig48fig} illustrates another torus $\TT^2\times \{-\tfrac\pi 4\}$ in $\TT^3$ and four more loops, $x_1,y_1,z_1,$ and $\mu_y'$.  
 Note that $x_1$ and $y_1$ are connected to the base point by a path which follows $z_1$ from $0$ to $\tfrac\pi 4$.   In conjunction with Figure \ref{fig26fig}, one sees easily that
  $x_1=z_1x \bar z_1$, $y_1=z_1y \bar z_1$,   $z_1=z$ and $\mu_y'=z_1\mu_y\bar z_1$
  in $\pi_1(Y^c,e)$.

We apply the SVK theorem to  the decomposition
  $Y^c=G\cup H$, with 
 $G=Y^c\cap \{-\tfrac\pi 4\leq z\leq 0\}\cup\partial \Nbd{C_z}$ and $
 H= Y^c\cap \{0\leq z\leq  \tfrac{7\pi} {4}\}.$
 The overlap $F:=G\cap H$ is a  genus two surface  with $$\pi_1(F,e)=\langle x,y,x_1,y_1\mid [y,\bar x]=[y_1,\bar x_1]\rangle.\hskip1.5in$$  
 
  \begin{wrapfigure}[8]{r}{0.25 \textwidth}
\begin{center}
\vskip-.3in
\includegraphics[width=2.2in]{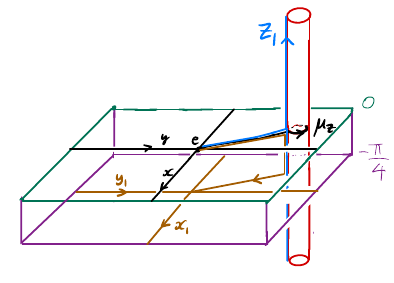}
 \vskip-.3in
\caption{\label{fig49fig}  }
\end{center}
\end{wrapfigure}

The subset $G$, depicted in Figure \ref{fig49fig},  deformation retracts onto a 2-complex which is the union of a torus with generators $\mu_z$ and $z_1$, and a punctured torus with generators $x,y$, attached to the first torus so that its boundary glues to $\mu_z$.   
It follows that
 $\pi_1(G,e)=\langle x,y,z_1\mid  [[y,\bar x],z_1]\rangle.$  
 The inclusion $\pi_1(F,e)\to \pi_1(G,e)$ is given by: 
$$x\mapsto x, y\mapsto y, x_1 \mapsto z_1x\bar z_1, y_1\mapsto z_1y\bar z_1.\hskip1.5in$$


  \begin{wrapfigure}[11]{r}{0.25\textwidth}
\begin{center}
\includegraphics[width=2.2in]{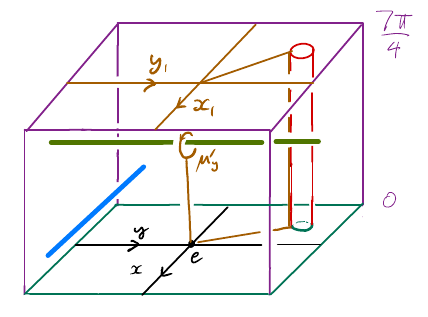}
 \vskip-.2in
\caption{\label{fig50fig}  }
\end{center}
\end{wrapfigure}
The subset $H$, depicted in Figure \ref{fig50fig},  deformation retracts onto the wedge of the two tori $\partial\Nbd{C_x}$ and $\partial\Nbd{C_y}$, and a moment's reflection on  the figure reveals that 
$$\pi_1(H,e)=\langle x,\mu_x,y_1,\mu_y' \mid [x,\mu_x],~[y_1, \mu_y' ]\rangle \hskip1.5in $$
and that 
the inclusion $\pi_1(F,e)\to\pi_1(H,e)$ is given by $$
x\mapsto x,~ y\mapsto y_1\mu_x, ~x_1 \mapsto \mu_y'   x , ~ y_1 \mapsto   y_1.\hskip1.5in$$

\vskip.3in

Hence the SVK theorem implies that $\pi_1(Y^c,e) $ has the presentation:  
\begin{align*}
\pi_1(Y^c,e) &=\langle
x,y,z_1,  \mu_x  , y_1,\mu_y'\mid [[y,\bar x],z_1],~[\mu_x,x],~[y_1,\mu_y' ], y=y_1\mu_x, z_1x\bar z_1=\mu_y'x, z_1 y \bar z_1=y_1
\rangle\\
&=  \langle x,y,z_1  \mid [[y,\bar x],z_1],       [[x, \bar z_1],y], [[\bar y,z_1],x]\rangle. 
\end{align*}
Since $z_1=z$ in $\pi_1(Y^c,e)$, the proof is complete. 
\qed

\section{Additional figures  for Section \ref{DA}}\label{AF}

To help the reader navigate the proofs in Section \ref{DA}, we provide a few more figures.

\begin{figure}[H]
\begin{center}
\vskip-.1in
\includegraphics[width=2in]{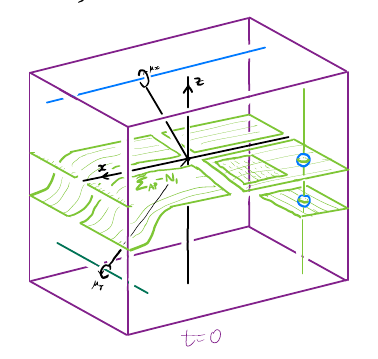}
\caption{\label{fig48fig} Proof of Lemma 10.4 (2): $x, z,\mu_x$ and $\mu_y$ miss $\tilde\Sigma_{AP,\delta}$. These loops lie in $\{t=0\}$, at distance greater than $\delta$ from $\tilde\Sigma_{AP}\setminus(N_1\cup N_2)$.}
\end{center}
\end{figure}

\begin{figure}[H]
\begin{center}
\vskip-.3in
\includegraphics[width=6.5in]{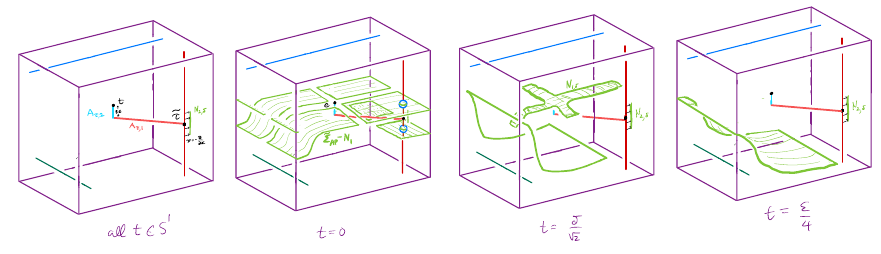}
\caption{\label{fig48fig} Proof of Lemma 10.6(c): $\tau=\mu_{R,\delta}^{k_1}t\mu_{R,\delta}^{-k_1}$ in $\pi_1(R\setminus\tilde\Sigma_{AP,\delta})$.  The annuli $A_{\tau,1}$ and $A_{\tau,2}$ provide a free homotopy in the complement of $\tilde\Sigma_{AP,\delta}$ from $\tilde \tau$, based at $(-\tfrac\pi 2, -\tfrac\pi 2,-3\ep,-3\ep)$, to $t$, based at $e$. There is a path on $A_{\tau,1}\cup A_{\tau,2}$ between these two points inside the $5\ep$ neighborhood of $\rho_1$. Also proves Lemma 10.4 (2): $t$ misses $\tilde\Sigma_{AP,\delta}$.}
\end{center}
\end{figure}

\begin{figure}[H]
\begin{center}
\vskip-.2in
\includegraphics[width=6.5in]{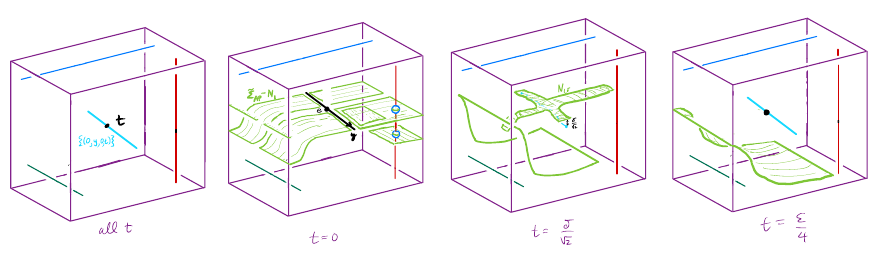}
\caption{\label{fig48fig} Proof of Lemma 10.6(d): $[t,y]=1$  and hence $[\tau,y]=1$ in $\pi_1(R\setminus\tilde\Sigma_{AP,\delta})$: the torus $\{(0,y,0,t)\}$ misses $\tilde\Sigma_{AP,\delta}$.}
\end{center}
\end{figure}

\begin{figure}[H]
\begin{center}
\includegraphics[width=2.5in]{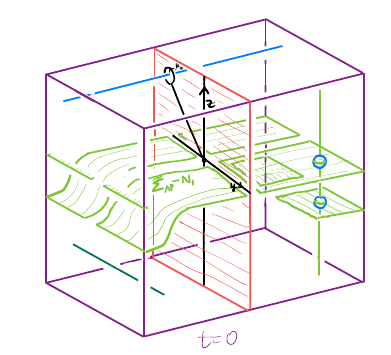}
\caption{\label{fig48fig} Proof of Lemma 10.6(e): $ \mu_x=[z,\bar y]$ in $\pi_1(R\setminus\tilde\Sigma_{AP,\delta})$.}
\end{center}
\end{figure}

\begin{figure}[H]
\centering
\hskip-1in\includegraphics[width=5in]{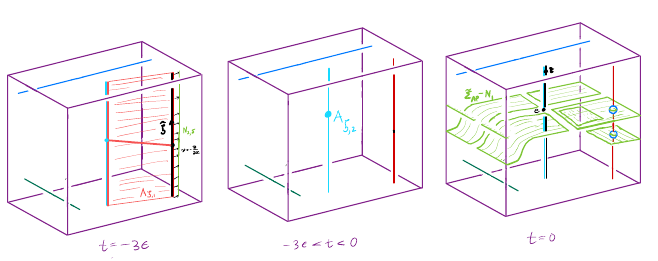}
\caption{\label{fig48fig} Proof of Lemma 10.6(f): $ 
\zeta=\mu_{R,\delta}^{k_2}   z \mu_{R,\delta}^{-k_2}$ in $\pi_1(R\setminus\tilde\Sigma_{AP,\delta})$. The annuli $A_{\zeta,1}$ and $A_{\zeta,2}$ determine a free homotopy from $\tilde\zeta$ to $z$ which misses $\tilde\Sigma_{AP,\delta}$. There is a path on $A_{\zeta,1}\cup A_{\zeta,2}$ between $e$ and the base point $(-\tfrac\pi 2, -\tfrac\pi 2,-3\ep,-3\ep)$ of $\tilde\zeta$ which lies inside the $5\ep$ neighborhood of $\rho_1$.}

\end{figure}

\end{document}